\documentclass[10pt,draft,reqno]{amsart}
   \usepackage{mathrsfs}
     \makeatletter
     \def\section{\@startsection{section}{1}%
     \z@{.7\linespacing\@plus\linespacing}{.5\linespacing}%
     {\bfseries
     \centering
     }}
     \def\@secnumfont{\bfseries}
     \makeatother
\newtheorem{theorem}{Theorem}[section]
\newtheorem{lemma}[theorem]{Lemma}
\newtheorem{proposition}[theorem]{Proposition}
\newtheorem{assumption}[theorem]{Assumption}

\newtheorem{definition}[theorem]{Definition}

\newtheorem{remark}[theorem]{Remark}
\newtheorem{hypothesis}[theorem]{Hypothesis}
\numberwithin{equation}{section} 
\DeclareMathOperator{\tr}{Tr}

\newcommand{\Tr}{\mathop{\mathrm{Tr}}}
\renewcommand{\d}{\/\mathrm{d}\/}

\def\w{w^{\varepsilon}_{{v}^{\varepsilon}}}

\def\e{\varepsilon}

\def\t{t\wedge\tau_N}
\def\T{T\wedge\tau_N}
\def\x{\textbf{X}}
\def\a{\textbf{A}}
\def\b{\textbf{B}}
\def\y{\textbf{Y}}
\def\z{\textbf{Z}}
\def\w{\textbf{W}}
\def\v{\textbf{v}}

\begin{document}

\title[2-D MHD systems with jump processes]{2-D Magneto-hydrodynamic system with jump processes: Well posedness and Invariant measures}

\author[U. Manna]{Utpal Manna} \author[M. T. Mohan]{Manil T.
Mohan}\address{School of Mathematics, Indian Institute of Science
Education and Research, Thiruvananthapuram, India}
\email{manna.utpal@iisertvm.ac.in, manil@iisertvm.ac.in}

\subjclass[2000] {Primary 60H15; Secondary 35Q35, 76D03, 76D06}

\keywords{Stochastic MHD system, L\'{e}vy processes, Local
monotonicity, Invariant measures.}

\begin{abstract}
In this work we prove the existence and uniqueness of the strong
solution to the two-dimensional stochastic magneto-hydrodynamic
system perturbed by L\'{e}vy noise. The local monotonicity arguments
have been exploited in the proofs. The existence of a unique
invariant measures has been proved using the exponential stability
of solutions.
\end{abstract}

\maketitle \noindent

\section{Introduction}
\setcounter{equation}{0}

Magneto-hydrodynamics (MHD) is the branch of fluid mechanics which
studies the motion of electrically conducted fluid in the presence
of magnetic field. MHD system consists of the Navier-Stokes
equations, which describe the motion of a fluid in the electric
field coupled with the Maxwell equations, which describe the motion
of the fluid in the magnetic field (see Chandrasekhar \cite{CH}).
The deterministic MHD system has been studied extensively by various
mathematicians and physicists (Cabannes \cite{Ca}, Cowling
\cite{Cw}, Ladyzhenskaya \cite{La}, Landau and Lifshitz \cite{LL},
Sermange and Temam \cite{ST}, and Temam \cite{Te1}, to name a few)
over the past fifty years. Significant contributions in the
mathematical study of stochastic MHD system are due to Sritharan and
Sundar \cite{SS3}, Sundar \cite{Su}, and Barbu and Da Prato
\cite{VB}. The authors in \cite{SS3} prove the well posedness of the
martingale problem associated to two and three-dimensional MHD
system perturbed by white noise. In \cite{Su}, the author study the
existence and uniqueness of solutions to the two-dimensional MHD
system perturbed by more general noise, namely multiplicative noise
and additive fractional Brownian noise. The authors in \cite{VB}
establish the existence and uniqueness of an invariant measure via
coupling methods developed by Odasso \cite{Od}.

In this paper we consider the stochastic MHD system in the
non-dimensional form. Let $G\subset \mathbb{R}^2$ be a bounded,
open, simply-connected domain with a sufficiently smooth boundary
$\partial G$, then
\begin{eqnarray*}
&&\frac{\partial u}{\partial t} + (u\cdot\nabla)u -
\frac{1}{R_e}\Delta u - S(B\cdot\nabla)B +
\nabla\left(p+\frac{S|B|^2}{2}\right) \nonumber \\ &&\qquad\quad =
\sigma_1(u,B)\partial W_1(t) + \int_Zg_1(u,z)\tilde{N_1}(dt,dz)
\textrm{ in }G\times(0,T),
\end{eqnarray*}
along with the Maxwell's equations
\begin{eqnarray*}
&&\frac{\partial B}{\partial t} + (u\cdot\nabla)B +
\frac{1}{R_m}\textrm{curl}(\textrm{curl}\: B) - (B\cdot\nabla)u
\nonumber \\ &&\qquad\quad = \sigma_2(u,B)\partial W_2(t) +
\int_Zg_2(u,z)\tilde{N_2}(dt,dz)\textrm{ in } G\times(0,T),
\end{eqnarray*}
where the terms and the initial and the boundary conditions are
discussed in section 2. We reduce the above coupled equations into a
single equation whose abstract form is given by
\begin{eqnarray*}
\d \x + [\a\x + \b(\x)]\d t = \sigma(t,\x)\d W(t) +
\int_Zg(\x,z)\tilde{N}(dt,dz)
\end{eqnarray*}
with $\x(0)=\x_0$, where $\x$ denotes the transpose of $(u,B)$. The
operators $\a$ and $\b$ are defined in section 2. $W(t)$ is an
$H$-valued Wiener process with positive symmetric trace class
covariance operator $Q$. $\tilde{N}(dt,dz) =N(dt,dz)-\lambda(dz)dt$
is a compensated Poisson random measure(cPrm), where $N(dt,dz)$
denotes the Poisson random measure associated to Poisson point
process $\mathbf{p}(t)$ on $Z$ and $\lambda(dz)$ is a
$\sigma$-finite L\'{e}vy measure on $(Z,\mathscr{B}(Z))$.

The construction of the paper is as follows. In the next section, we
give the description of the deterministic MHD system and describe
the functional setting of the problem. In section $3$, we first
formulate the abstract stochastic model when the noise coefficients
are small and then prove certain a-priori energy estimates (both
$L^2$ and $L^p$) with exponential weights.  We then prove the
existence and uniqueness of the strong solution exploiting certain
monotonicity property of the operators. In section $4$, existence of
a unique invariant measures has been established using the tightness
property and the exponential stability of solutions.

\section{The Deterministic MHD System}
\setcounter{equation}{0}

In this section we first describe the deterministic MHD System with
suitable functional setting for the model. Also we will discuss
about the linear and non-linear operators and their properties. Some
of the results of this section are taken from Sritharan and Sundar
\cite{SS3}, Sundar \cite{Su}, and Temam \cite{Te1}. Some results
presented in this section are known but given for the sake of
completeness.

\subsection{The Description of the Deterministic MHD System}

 We consider the motion of a viscous incompressible
and resistive fluid filling $G$, where $G\subset\mathbb{R}^2$ is a
bounded, open, simply-connected, domain with a sufficiently smooth
boundary $\partial G$. We fix a final time $T>0$. The motion
described by the MHD system (see Temam \cite{Te1}) in the
non-dimensional form is given by
\begin{align}\label{magneto}
\frac{\partial u}{\partial t} + (u\cdot\nabla)u -
\frac{1}{R_e}\Delta u - S(B\cdot\nabla)B +
\nabla\left(p+\frac{S|B|^2}{2}\right)= f_1(t)\textrm{ in }
G\times(0,T),
\end{align}
where $p$ denotes the pressure of the fluid, $u(t,x) =
(u_1(t,x),u_2(t,x))$ the velocity of the particle of fluid which is
at the point $x$ at time $t$, $B(t,x)=(B_1(t,x),B_2(t,x))$ the
magnetic field at a point $x$ at time $t$ and $R_e$ the Reynolds
number. Let $S = \frac{M^2}{R_eR_m}$, where $M$ is the Hartman
number and $R_m$ is the magnetic Reynolds number, $\Delta$ denotes
the Laplace operator and $\nabla$ the gradient. The Maxwell's
equation is given by
\begin{eqnarray}\label{maxwell}
\frac{\partial B}{\partial t} &+ (u\cdot\nabla)B +
\frac{1}{R_m}\textrm{curl}(\textrm{curl}\: B) - (B\cdot\nabla)u =
f_2(t)\textrm{ in }G\times(0,T),
\end{eqnarray}
where
\begin{eqnarray}\label{condition}
\nabla\cdot u = 0\textrm{ and } \nabla\cdot B = 0 \textrm{ in }
G\times(0,T),
\end{eqnarray}
where $f_1(t)$ and $f_2(t)$ are the external forces acting on the
system.

The above system of equations deduced from the coupled equation of
Navier-Stokes and the Maxwell's equations. The Navier-Stokes part is
\begin{eqnarray*}
\frac{\partial u}{\partial t}+(u\cdot \nabla) u=\frac{j\times
B}{\rho}+\frac{qE}{\rho}-\frac{\nabla p}{\rho}+\nu\Delta u,\;\;
\nabla\cdot u=0,
\end{eqnarray*}
where $E$ is the electric field, $\rho$ is the mass density, $-q$ is
the charge of an electron and $\nu$ is the viscosity. Let $\mu_0$ be
the permeability of the free space, $\rho_c$ be the charge density
and $c$ be the speed of light, the Maxwell's part is
\begin{align*}
\frac{\partial B}{\partial t}=-\mathrm{curl} E,\;\; \mathrm{curl}
B=\mu_0 j+\frac{1}{c^2}\frac{\partial E}{\partial t},\;\;
\nabla\cdot B=0,\;\;\nabla\cdot E=\rho_c.
\end{align*}
But in most of the practical applications, only low-frequency
behavior is studied and thus the term $\frac{q}{\rho}E$ is ignored
from the Navier-Stokes part and the displacement current,
$\frac{1}{c^2}\frac{\partial E}{\partial t}$ from the Maxwell's
part. A resistive form of Ohm's law is used to eliminate $E$, the
electric field from the equations, $j=\sigma(E+u\times B)$, where
$\sigma$ denotes the electric conductivity of the fluid. (For more
details see Chapter 8 of Landau, Lifshitz, and Pitaevskii \cite{LLP}
and Section 7.7 of Jackson \cite{JJ}.) For a vector-valued function
$u = (u_1,u_2)$ defined on $G$, $\textrm{curl}\: u = \frac{\partial
u_2}{\partial x} - \frac{\partial u_1}{\partial y}.$ Let $\phi$ be a
scalar-valued function defined on $G$, then $\textrm{curl}\: \phi =
\left(\frac{\partial\phi}{\partial y},-\frac{\partial\phi}{\partial
x}\right).$ Also we know that $\textrm{curl}(\textrm{curl}\: u) =
\textrm{grad}\:\textrm{div}\:u - \Delta u$ is valid in the two
dimensions. Now let us give the boundary conditions for the
equations \eqref{magneto} and \eqref{maxwell}.
\begin{eqnarray}\label{bound1}
u(t,x) = 0\textrm{ on } (0,T) \times \partial G,\\
B\cdot\hat{n} = 0  \textrm{ and } \textrm{curl}\:B = 0 \textrm{ on }
(0,T)\times\partial G.\label{bound2}
\end{eqnarray}
The initial conditions are given by
\begin{eqnarray}\label{initial}
u(x,0) = u_0(x),\quad B(x,0) = B_0(x) \textrm{ for all } x\in G.
\end{eqnarray}
\noindent Here $\hat{n}$ denotes the unit outward normal to
$\partial G$. Also \eqref{bound1} is the nonslip condition and
\eqref{bound2} describes a perfectly conducting wall. The system
\eqref{magneto}-\eqref{initial} needs to be written as an abstract
evolution equation with
\begin{eqnarray}\label{equation1}
\x(t) = \x(0) + \int_0^t\left[-\a\x(s)-\b\left(\x(s)\right)\right]\d
s +\int_0^t \textbf{f}(t)dt,
\end{eqnarray}
for all $t\in[0,T]$. In the above formulation, $\x$ denotes the
transpose of $(u,B)$ and $\textbf{f}(t)$ is the transpose of
$\{f_1(t),f_2(t)\}$.

\subsection{Abstract Formulation of the Equations}
In this section we collect the needed information for the operator
formulation of the problem \eqref{magneto}-\eqref{initial} as in
Sermanage and Temam \cite{ST}. We consider the spaces $H = H_1 \times H_2$ and $V = V_1 \times V_2$, where
\begin{eqnarray*}
H_1 &=& \left\{\phi \in \textrm{L}^2(G) : \nabla \cdot \phi =
0,\;\; \phi\cdot\hat{n}\big|_{\partial G} = 0\right\},
H_2 = H_1,\\
V_1 &=& \left\{\phi \in H_0^1(G) : \nabla\cdot\phi = 0\right\}, \
V_2 = \left\{\phi \in H^1(G) : \nabla \cdot \phi = 0,
\phi\cdot\hat{n}\big|_{\partial G} = 0\right\},\\
H^1 &=&\left\{u\in \textrm{L}^2(G) : \nabla u \in
\textrm{L}^2(G)\right\},H_0^1 = \left\{u \in \textrm{L}^2(G) :
\nabla u \in \textrm{L}^2(G), u\big|_{\partial G} = 0\right\}.
\end{eqnarray*}
\noindent Here $H_1$ and $H_2$ are equipped with the
$\textrm{L}^2(G)$ norm. Let us define the inner product on $H$ by
$[\x_1,\x_2] = (u_1,u_2)_{H_1} + S(B_1,B_2)_{H_2}, \textrm{ where
}\x_i = \{u_i,B_i\}.$ Hence we get the inner product is equivalent
to $(\x_1,\x_2)_H = (u_1,u_2)_{H_1} + (B_1,B_2)_{H_2}$ and the norm
on $H$ is given by $\|\x\|_H = \sqrt{(\x,\x)_H}.$ Also the space
$V_1$ is endowed with the inner product given by
$[[\phi,\psi]]_{V_1} = (\nabla\phi,\nabla\psi)_{\textrm{L}^2(G)}.$
The norm on $V_1$ is given by $\|\phi\|_{V_1} =
\sqrt{[[\phi,\phi]]_{V_1}}$. Let us note that $V_1$ norm given here
is equivalent to the usual $H^1(G)$ norm since
$\|\nabla\phi\|_{\textrm{L}^2(G)} \geq C\|\phi\|_{\textrm{L}^2(G)}$
for a suitable $C$ by the Poincar\'{e} inequality. Let $V_2$ be
endowed with the inner product $[[\phi,\psi]]_{V_2} = (\textrm{curl}
\phi,\textrm{curl} \psi)_{\textrm{L}^2(G)}$ and the norm on $V_2$ is
given by $\|\phi\|_{V_2} = \sqrt{[[\phi,\phi]]_{V_2}}.$ From
Proposition 1.8 and Lemma 1.6 of Temam \cite{Te}, we conclude that
the norm given by $\left\{\|\phi\|^2_{\textrm{L}^2(G)} +
\|\textrm{curl}\:\phi\|^2_{\textrm{L}^2(G)} +
\|\phi\cdot\hat{n}\|^2_{H^{1/2}(\partial G)}\right\}^{1/2}$ is
equivalent to the $H^1(G)$ norm. Thus the space $V$ is endowed with
the scalar product: $[[\x_1,\x_2]] = [[u_1,u_2]]_{V_1} +
S[[B_1,B_2]]_{V_2}.$ Having defined the spaces $H$ and $V$, we have
the dense, continuous and compact embedding: $V\hookrightarrow
H\equiv H^\prime \hookrightarrow V^\prime.$

\subsection{The linear operator}
In this section we define the operator $\a$ that appear in
$\eqref{equation1}$ and discuss its properties. $\a$ is defined
through a bilinear coercive form.

Let us define a function $a : V \times V \rightarrow \mathbb{R}$ as
follows:
\begin{eqnarray}\label{linear}
a(\x_1,\x_2) = \frac{1}{R_e}[[u_1,u_2]]_{V_1} +
\frac{S}{R_m}[[B_1,B_2]]_{V_2}.
\end{eqnarray}

\begin{proposition}
The function $a(\cdot,\cdot)$ defined by \eqref{linear} is
continuous and coercive.
\end{proposition}

\begin{proof}
Without loss of generality let us assume that $R_e$ and $R_m$ be
equal to 1. Then there exists a constant $k$ such that
\begin{eqnarray*}
&&|a(\x_1,\x_2)| = \Big|[[u_1,u_2]]_{V_1} + S[[B_1,B_2]]_{V_2}\Big|\\
&&\qquad\qquad\quad\leq \|\nabla u_1\|_{\textrm{L}^2(G)} \|\nabla
u_2\|_{\textrm{L}^2(G)} + S\|\textrm{curl}\:B_1\|_{\textrm{L}^2(G)}
\|\textrm{curl}\:B_2\|_{\textrm{L}^2(G)}\\
&&\qquad\qquad\quad\leq k\Big(\|u_1\|_{H^1(G)} \|u_2\|_{H^1(G)} + S\|B_1\|_{H^1(G)} \|B_2\|_{H^1(G)}\Big)\\
&&\qquad\qquad\quad\leq k\Big(\|u_1\|^2_{H^1(G)} + \|B_1\|^2_{H^1(G)}\Big)^{1/2} \Big(\|u_2\|^2_{H^1(G)} + \|B_2\|^2_{H^1(G)}\Big)^{1/2}\\
&&\qquad\qquad\quad\leq k\|\x_1\|_V\|\x_2\|_V.
\end{eqnarray*}
Hence we obtain $|a(\x_1,\x_2)|\leq k\|\x_1\|_V\|\x_2\|_V\;
\textrm{for all }\;\x_1,\x_2\in V$. For the coercive property we use
the estimate $\frac{1}{k}\|u\|_{H^1(G)}\leq
\|u\|_{\textrm{L}^2(G)}\leq k\|u\|_{H^1(G)}$, consider
\begin{eqnarray*}
&&a(\x,\x) = [[u,u]]_{V_1} + S[[B,B]]_{V_2}= \|\nabla u\|^2 + S\|\textrm{curl}\:B\|^2\\
&&\qquad\quad\;\;\geq c\Big(\|u\|^2_{H^1(G)} +
\|B\|^2_{H^1(G)}\Big)=c\|\x\|^2_V.
\end{eqnarray*}
Hence we get for a positive constant $c$, $a(\x,\x) \geq
c\|\x\|^2_V$ for all $\x\in V$.
\end{proof}

By Lax-Milgram Theorem, there exists an operator $\a : V \rightarrow
V^\prime$ such that
\begin{eqnarray}\label{lin}
a(\x,\y) = \big(\a\x,\y\big)_{V\times V^\prime}\textrm{ for all }
\x,\y \in V.
\end{eqnarray}
Therefore $\a : V \rightarrow V^\prime$ can be restricted to a
self-adjoint operator $\a : \mathfrak{D}(\a) \rightarrow H$. Now we
can write $\mathfrak{D}(\a) = \mathfrak{D}(\a_1) \times
\mathfrak{D}(\a_2)$ where $\a_1, \a_2, \mathfrak{D}(\a_1),
\mathfrak{D}(\a_2)$ are obtained as follows:

Let us consider the ``Stokes" problem in $G$:
$$\frac{-1}{R_e}\Delta u + \nabla p = g  \textrm{ with } \nabla\cdot u = 0\textrm{ and } u\big|_{\partial G} = 0.$$
Let $\phi \in V_1$ and multiply the above equation with $\phi$.
Apply the integration by parts formula to obtain
$\frac{1}{R_e}[[u,\phi]]_{V_1} = (g,\phi)$ for $\phi \in V_1$.
Define $\a_1u = \frac{1}{R_e}[[u,\cdot]]_{V_1} = g \in V_1^\prime$.
By the Cattabriga regularity theorem, we can conclude that
$\mathfrak{D}(\a_1) = H^2(G)\cap V_1$ whenever $g\in H_1$.

Consider the elliptic problem in $G$ for defining $\a_2$ :
$$\frac{1}{R_m}\textrm{curl}\:\textrm{curl}\:B = g  \textrm{ with } \nabla\cdot B;\quad B\cdot\hat{n}\big|_{\partial G} = 0;
\quad \textrm{curl B}\big|_{\partial G} = 0,$$ so that $\a_2B =
\frac{1}{R_m}\textrm{curl}\:\textrm{curl B} = g.$ It is proved in
\cite{ST} that if $g \in H$, then
$$\mathfrak{D}(\a_2) = \big\{u \in H^2(G) : \nabla\cdot u = 0, u \cdot\hat{n}\big|_{\partial G} = 0
\textrm{ and } \textrm{curl u}\big|_{\partial G} = 0\big\}.$$

\subsection{The nonlinear operator}
We next define $\b$ that figures in \eqref{equation1}. Let us
consider the trilinear form $b(\cdot,\cdot,\cdot) : V \times V
\times V \rightarrow \mathbb{R}$ defined by
\begin{align}\label{nonlinear}
b(\x_1,\x_2,\x_3) = \tilde{b}(u_1,u_2,u_3) - S\tilde{b}(B_1,B_2,u_3)
+ S\tilde{b}(u_1,B_2,B_3) - S\tilde{b}(B_1,u_2,B_3),
\end{align}
for all $\x_i = (u_i,B_i) \in V$ where $\tilde{b}(\cdot, \cdot,
\cdot) : (H^1(G))^{\otimes3} \rightarrow \mathbb{R}$ is defined by
$$\tilde{b}(\phi, \psi, \theta) = \left(\sum_{i,j=1}^2\int_G\phi_i\frac{\partial\psi_j}{\partial x_i}\theta_j\d x\right).$$

\begin{proposition}
The function $b(\cdot,\cdot,\cdot)$ defined by \eqref{nonlinear} is
continuous.
\end{proposition}

\begin{proof}
In order to prove the continuity of $b$, we need to show the
continuity of $\tilde{b}$. Let us first consider
$\tilde{b}(\phi,\psi,\theta)$ and by using the H\"{o}lder's
inequality, we have
\begin{eqnarray*}
|\tilde{b}(\phi, \psi, \theta)| =
\left|\sum_{i,j}\int_G\phi_i\frac{\partial\psi_j}{\partial
x_i}\theta_j\d x\right|\leq \|\phi\|_{\textrm{L}^4(G)}
\|\nabla\psi\|_{\textrm{L}^2(G)}\|\theta\|_{\textrm{L}^4(G)}.
\end{eqnarray*}
Since $H^1(G) \subset \textrm{L}^4(G)$, we get $|\tilde{b}(\phi,
\psi, \theta)| \leq c \|\phi\|_{H^1} \|\psi\|_{H^1}
\|\theta\|_{H^1}.$ Now,
\begin{eqnarray*}
&&|b(\x_1, \x_2, \x_3)|\\
&&\leq |\tilde{b}(u_1,u_2,u_3)| + S|\tilde{b}(B_1,B_2,u_3)| + S|\tilde{b}(u_1,B_2,B_3)|  + S|\tilde{b}(B_1,u_2,B_3)|\\
&&\leq c\Big(\|u_1\|_{H^1} \|u_2\|_{H^1} \|u_3\|_{H^1} + \|B_1\|_{H^1} \|B_2\|_{H^1}\|u_3\|_{H^1} \\
&&\quad + \|u_1\|_{H^1} \|B_2\|_{H^1} \|B_3\|_{H^1} + \|B_1\|_{H^1} \|u_2\|_{H^1}\|B_3\|_{H^1}\Big)\\
&&\leq c\Big(\|u_1\|^2_{H^1(G)} + \|B_1\|^2_{H^1(G)}\Big)^{1/2}\Big(\|u_2\|^2_{H^1(G)} + \|B_2\|^2_{H^1(G)}\Big)^{1/2}\\
&&\qquad\Big(\|u_3\|^2_{H^1(G)} + \|B_3\|^2_{H^1(G)}\Big)^{1/2}\\
&&\leq c\|\x_1\|_V\|\x_2\|_V\|\x_3\|_V,
\end{eqnarray*}
since $\|\x\|_V$ is equivalent to $\Big(\|u\|^2_{H^1(G)} +
\|B\|^2_{H^1(G)}\Big)^{1/2}.$
\end{proof}
Note that $\tilde{b}(\phi, \psi, \theta) = -\tilde{b}(\phi, \theta,
\psi)$ if $\psi$ or $\theta \in V_1$ and $\phi \in H_1$. Therefore
we have $b(\x,\y,\z)=-b(\x,\z,\y)$ by using (\ref{nonlinear}). Also
note that for all $\phi \in H^1(G)$ and $\psi \in V_1$, $
\tilde{b}(\phi, \psi, \psi) =
\sum_{i,j=1}^2\int_G\phi_i\frac{\partial\psi_j}{\partial
x_i}\psi_j\d x = 0, $ by integration by parts. Therefore by
(\ref{nonlinear}), we get, $b(\x,\y,\y)=0$. We now define $\b : V
\times V\rightarrow V^\prime$ as the continuous bilinear operator
such that $b(\x_1,\x_2,\x_3) = \big(\b(\x_1,\x_2),\x_3\big) \textrm{
for all } \x_1, \x_2, \x_3 \in V.$ Let the $H$-norm of $\x$ be
denoted as $|\x|$ and the $V$-norm of $\x$ be denoted as $\|\x\|$.
The existence of the operator $\b$ is justified by the Riesz
representation theorem and let $\b(\x)$ denote $\b(\x,\x)$. It is
well known that for any $\x \in V,$ $\|\b(\x)\|_{V'}\leq
C|\x|\|\x\|.$ Since the constant $C$ do not play a crucial role in
this paper, we will set $C = 1$.

\section{Stochastic MHD System with L\'{e}vy noise}
\setcounter{equation}{0}

\subsection{Basic Concepts}
In this sub-section definitions of Hilbert space valued Wiener
processes and L\'{e}vy processes have been presented. Interested
readers may look into Da Prato and Zabczyk \cite{DaZ}, Applebaum
\cite{Ap}, Peszat and Zabczyk \cite{PZ}, Kingman \cite{Km} for extensive study on
the subject.

\begin{definition}
Let $H$ be a Hilbert space. A stochastic process $\{W(t)\}_{0\leq
t\leq T}$ is said to be an $H$-valued $\mathscr{F}_t$-adapted Wiener
process with covariance operator $Q$ if
\begin{enumerate}
\item [(i)] For each non-zero $h\in H$, $|Q^{1/2}h|^{-1} (W(t), h)$ is a standard one - dimensional Wiener process,
\item [(ii)] For any $h\in H, (W(t), h)$ is a martingale adapted to $\mathscr{F}_t$.
\end{enumerate}
\end{definition}
If $W$ is a an $H$-valued Wiener process with covariance operator
$Q$ with $\Tr Q < \infty$, then $W$ is a Gaussian process on $H$ and
$ \mathbb{E}(W(t)) = 0,$ $\text{Cov}\ (W(t)) = tQ,$ $t\geq 0.$ Let
$H_0 = Q^{1/2}H.$ Then $H_0$ is a Hilbert space equipped with the
inner product $(\cdot, \cdot)_0$, $(u, v)_0 = \left(Q^{-1/2}u,
Q^{-1/2}v\right),\ \forall u, v\in H_0,$ where $Q^{-1/2}$ is the
pseudo-inverse of $Q^{1/2}$. Since $Q$ is a trace class operator,
the imbedding of $H_0$ in $H$ is Hilbert-Schmidt. Let $L_Q$ denote
the space of linear operators $S$ such that $S Q^{1/2}$ is a
Hilbert-Schmidt operator from $H$ to $H$.

Let $\mathbb{M}$ be the totality of non-negative (possibly infinite)
integral valued measures on $(H,\mathscr{B}(H))$ and $\mathscr{B}_{\mathbb{M}}$
be the smallest $\sigma$-field on $\mathbb{M}$ with respect to
which all $N\in\mathbb{M}\to N(B)\in
\mathbb{Z}^+\cup\{\infty\}, B\in\mathscr{B}(H)$, are measurable.
\begin{definition}
An $(\mathbb{M},\mathscr{B}(\mathbb{M}))$-valued random variable
$N$ is called a \textit{Poisson random measure}
\begin{enumerate}
\item if for each $B\in\mathscr{B}(H), N(B)$ is Poisson distributed.
i.e., $\mathbb{P}(N(B) = n)
=\eta(B)e^{-\eta(B)}/n!,n=0,1,2,\ldots,$ where
$\eta(B)=\mathbb{E}(N(B)), B\in\mathscr{B}(H);$
\item if $B_1,B_2,\ldots,B_n\in\mathscr{B}(H)$ are disjoint, then
$N(B_1), N(B_2),\ldots,N(B_n)$ are mutually independent.
\end{enumerate}
\end{definition}
A \textit{point function} $\mathbf{p}$ on $H$ is a mapping
$\mathbf{p}:D_{\mathbf{p}}\subset (0,\infty)\to H,$ where
the domain $D_{\mathbf{p}}$ is a countable subset of
$(0,\infty)$. $\mathbf{p}$ defines a counting measure
$N_{\mathbf{p}}$ on $(0,\infty)\times H$ by
$N_{\mathbf{p}}((0,t]\times Z, \omega)=\#\left\{s\in
D_{\mathbf{p}}(\omega); s\leq t, \mathbf{p}(s, \omega)\in Z\right\},
t>0, Z\in\mathscr{B}(H), \omega\in\Omega.$ Let $\Pi_{H}$ be the totality of point
functions on $H$ and $\mathscr{B}(\Pi_H)$ be the smallest $\sigma$-field on $\Pi_{H}$ with respect to which all $\mathbf{p}\to
N_{\mathbf{p}}((0,t]\times Z), t>0, Z\in \mathscr{B}(H)$, are
measurable.
\begin{definition}
A point process $\mathbf{p}$ on $H$ is a
$(\Pi_H,\mathscr{B}(\Pi_H))$-valued random variable, that is, a
mapping $\mathbf{p}:\Omega\to \Pi_H$ defined on a probability space
$(\Omega,\mathscr{F},\mathbb{P})$ which is
$\mathscr{F}/\mathscr{B}(\Pi_H)$-measurable. A point process
$\mathbf{p}$ is called \textit{Poisson point process} if $N_{\mathbf{p}}(\cdot)$, as defined above, is a Poisson random measure on $(0,\infty)\times H.$
\end{definition}
L\'{e}vy process is a special type of Poisson random measure associated to a Poisson point process.
\begin{definition}
A c\`{a}dl\`{a}g adapted process, $(\mathbf{L}_t)_{t\geq 0}$, is called a L\'{e}vy process if it
has stationary independent increments and is stochastically
continuous.
\end{definition}
Let $(\mathbf{L}_t)_{t\geq 0}$ be a $H$-valued L\'{e}vy process.
Hence, for every $\omega\in\Omega$, $\mathbf{L}_t(\omega)$ has
countable number of jumps on $[0, t]$. Note that for every
$\omega\in\Omega$, the jump $\triangle \mathbf{L}_t(\omega) =
\mathbf{L}_t(\omega) - \mathbf{L}_{t-}(\omega)$ is a point function
in $\mathscr{B}(H\backslash\{0\})$. Let us define $N(t, Z)= N(t, Z,
\omega)= \#\left\{s\in (0, \infty): \triangle\mathbf{L}_s(\omega)\in
Z\right\}, t>0, Z\in\mathscr{B}(H\backslash\{0\}), \omega\in\Omega$
as the \emph{Poisson random measure associated with the L\'{e}vy
process} $(\mathbf{L}_t)_{t\geq 0}$.

The differential form of the measure $N(t,Z,\omega)$ is written as
$N(dt, dz)(\omega)$. We call $\tilde{N}(dt, dz) = N(dt, dz) -
\lambda(dz)dt $ a \emph{compensated Poisson random measure (cPrm)},
where $\lambda(dz)dt $ is known as \emph{compensator} of the
L\'{e}vy process $(\mathbf{L}_t)_{t\geq 0}$. Here $dt$ denotes the Lebesgue
measure on $\mathscr{B}(\mathbb{R}^{+})$, and $\lambda(dz)$ is a
$\sigma$-finite L\'{e}vy measure on $(Z, \mathscr{B}(Z))$.
\begin{definition}
A non-empty collection of sets $\mathscr{G}$ is a semi-ring if
\begin{itemize}
\item [(i)] Empty set $\phi\in\mathscr{G}$;
\item [(ii)] If $A\in\mathscr{G},$ $B\in\mathscr{G}$, then $A\cap
B\in\mathscr{G}$;
\item [(iii)] If $A\in\mathscr{G},$ $A\supset A_1\in\mathscr{G}$,
then $A=\cup_{k=1}^nA_k$, where $A_k\in\mathscr{G}$ for all $1\leq
k\leq n$ and $A_k$ are disjoint sets.
\end{itemize}
\end{definition}
\begin{definition}
Let $H$ and $F$ be separable Hilbert spaces. Let $F_t:=
\mathscr{B}(H)\otimes\mathscr{F}_t$ be the product $\sigma$-algebra
generated by the semi-ring $ \mathscr{B}(H)\times\mathscr{F}_t$ of
the product sets $Z\times F,$ $Z\in\mathscr{B}(H),$ $F\in
\mathscr{F}_t$ (where $\mathscr{F}_t$ is the filtration of the
additive process $(\mathbf{L}_t)_{t\geq 0}$). Let $T>0$, define
\begin{align*}
\mathbb{H}(Z) = & \big\{g : \mathbb{R}^+ \times Z \times \Omega
\rightarrow F,\textrm{ such that g is}\;\; F_T/\mathscr{B}(F) \;\;
measurable\;\;and \nonumber\\& \;\;\qquad g(t,z,\omega)\;\;is\;\;
\mathscr{F}_t - adapted\;\;\forall z\in Z, \forall t\in (0,T]\big\}.
\end{align*}
For $p\geq1$, let us define,
$$\mathbb{H}^{p}_\lambda([0,T]\times Z;F) = \left\{g\in \mathbb{H}(Z) :
 \int_0^T\int_Z\mathbb{E}[\|g(t,z,\omega)\|^p_F]\lambda(dz)dt < \infty \right\}.$$
\end{definition}
For more details see Mandrekar and R\"{u}diger \cite{MR}.

\begin{remark}
Let us define $\mathscr{D}([0,T];H)$ as the space of all
c\'{a}dl\'{a}g paths from $[0,T]$ into $H$, where $H$ is a Hilbert
space.
\end{remark}

\begin{lemma}\label{lemma100}(Kunita's Inequality)
Let us consider the stochastic differential equations driven by
L\'{e}vy noise of the form
\begin{eqnarray*}
\x(t)=\int_0^t b(\x(s))\d s+\int_0^t\sigma(s,\x(s))\d
W(s)+\int_0^t\int_Zg(\x(s-),z)\tilde{N}(ds,dz).
\end{eqnarray*}
Then for all $p\geq 2$, there exists $C(p,t)>0$ such that for each
$t\geq 0,$
\begin{align*}
&\mathbb{E}\left[\sup_{t_0\leq s\leq
t}|\x(s)|^p\right]\nonumber\\&\leq
C(p,t)\left\{\mathbb{E}|\x(0)|^p+\mathbb{E}\left(\int_{0}^t|b(\x(r))|^p\d
r\right)+\mathbb{E}\left[\int_{0}^t|\sigma(r,\x(r))|^p\d
r\right]\right.\\&\left.\;+\mathbb{E}\left[\int_{0}^t\left(\int_Z|g(\x(r-),z)|^2\lambda(dz)\right)^{p/2}\d
r\right]+\mathbb{E}\left[\int_{0}^t\int_Z|g(\x(r-),z)|^p\lambda(dz)dr\right]\right\}.
\end{align*}
\end{lemma}
\noindent For proof see Corollary 4.4.24 of \cite{Ap}.

\subsection{Formulation of the Model}
The stochastic MHD system perturbed by L\'{e}vy noise is given by
\begin{eqnarray}\label{magneto1}
&&\frac{\partial u}{\partial t} + (u\cdot\nabla)u -
\frac{1}{R_e}\Delta u - S(B\cdot\nabla)B +
\nabla\left(p+\frac{S|B|^2}{2}\right) \nonumber \\ &&\quad =
\sigma_1(u,B)\partial W_1(t) + \int_Zg_1(u,z)\tilde{N_1}(dt,dz)\quad
\textrm{in}\quad G\times(0,T),
\end{eqnarray}
\begin{eqnarray}\label{maxwell1}
&&\frac{\partial B}{\partial t} + (u\cdot\nabla)B +
\frac{1}{R_m}\textrm{curl}(\textrm{curl}\: B) - (B\cdot\nabla)u
\nonumber \\ &&\quad = \sigma_2(u,B)\partial W_2(t) +
\int_Zg_2(u,z)\tilde{N_2}(dt,dz)\quad \textrm{in}\quad G\times(0,T),
\end{eqnarray}
where $Z$ is a measurable space on $H$ (recall that $H=H_1\times
H_2$). The processes $W_1$ and $W_2$ are independent Wiener
processes and $\tilde{N_1}(dt,dz) = N_1(dt,dz) - \lambda_1(\d z)\d
t$ and $\tilde{N_2}(dt,dz) = N_2(dt,dz) - \lambda_2(\d z)\d t$ are
compensated Poisson random measures.

Next we introduce the operators $\sigma = \left(
  \begin{array}{cc}
    \sigma_1 & 0 \\
    0 & \sigma_2 \\
  \end{array}
\right)  \textrm{ and }  g =\left(
       \begin{array}{cc}
         g_1 & 0 \\
         0 & g_2 \\
       \end{array}
     \right)$
and the vectors $W = \left(
                       \begin{array}{c}
                         W_1 \\
                         W_2 \\
                       \end{array}
                     \right),$
 $\tilde{N} = \left(
                          \begin{array}{c}
                            \tilde{N_1} \\
                            \tilde{N_2} \\
                          \end{array}
                        \right)$ and $\lambda=\left(\begin{array}{c}\lambda_1\\ \lambda_2\end{array}\right)$.
The noise term in its integral form is given by
$\int_0^T\sigma(r,\x(r))\d W(r) +
\int_0^t\int_Zg(\x(r),z)\tilde{N}(dr,dz)$, where $W$ is an
$H$-valued Wiener process with a nuclear covariance form $Q$ and
$\tilde{N}(dt,dz)$ is a compensated Poisson random measure. By the
above description and the abstract formulation of the model,
(\ref{magneto1})-(\ref{maxwell1}) in the integral form as
\begin{eqnarray}\label{equation123}
&&\x(t) = \x(0) + \int_0^t\left[-\a\x(s)-\b\left(\x(s)\right)\right]\d s + \int_0^t\sigma(s,\x(s))\d W(s) \nonumber \\
&&\qquad\qquad\qquad\qquad\qquad\qquad\qquad +
\int_0^t\int_Zg(\x(s-),z)\tilde{N}(ds,dz).
\end{eqnarray}
Note that in the above formulation $\x$ denotes the transpose of
$(u,B)$.

\subsection{Hypothesis and Local Monotonicity}

Assume that $\sigma$ and $g$ satisfy the following hypothesis of
joint continuity, Lipschitz condition and linear growth.
\begin{hypothesis}\label{hyp}
The main hypothesis is the following,
\begin{itemize}
    \item[(H.1)]  The function $\sigma \in C([0, T] \times V; L_Q(H_0; H))$, and $g\in \mathbb{H}^2_{\lambda}([0, T] \times Z; H)$.

    \item[(H.2)]  For all $t \in (0, T)$, there exists a positive constant $K$ such that for all $\x \in H$,
    $$|\sigma(t, \x)|^2_{L_Q} + \int_{Z} |g(\x, z)|^2_{H}\lambda(dz) \leq K(1 +|\x|^2).$$

    \item[(H.3)]  For all $t \in (0, T)$,  there exists a positive constant $L<1$ such that for all $\x, \y \in H$,
    $$|\sigma(t, \x) - \sigma(t, \y)|^2_{L_Q} + \int_{Z} |g(\x, z)-g(\y, z)|^2_{H}\lambda(dz)\leq L|\x - \y|^2.$$
\end{itemize}
\end{hypothesis}
\begin{remark}
From the above hypothesis, we have $$\sigma_i \in C([0, T] \times
V_i; L_Q(H_{i_0}; H_i))\textrm{ and }g_i\in
\mathbb{H}^2_{\lambda_i}([0, T] \times Z_i; H_i)\textrm{ for
}i=1,2.$$
\end{remark}
\begin{remark}
Notice that, $Q:H\rightarrow H$ is a trace class covariance
(nuclear) operator and hence compact. So $H_0=Q^{1/2}H$ is a
separable Hilbert space and the imbedding of $H_0$ in $H$ is
Hilbert-Schmidt. Let $\{e_n\}_{n=1}^{\infty}$ be the eigenfunctions
of $Q$ (may not be complete). Then $Qe_n=\beta_ne_n$, where each
$\beta_n$ is positive real and $\sum_n\beta_n<\infty$. Let
$\{h_m\}$, with $h_m=\sqrt{\beta_m}e_m,m=1,2,\cdots$ be orthonormal
basis in $H_0$ (see section 4.1, chapter 4 of Da Prato and Zabczyk
\cite{DaZ}). Then,
\begin{eqnarray*}
|\sigma(t,\x)|^2_{\mathrm{L}_Q}&=&\sum_{m,n=1}^{\infty}|(\sigma
h_m,e_n)|^2=\sum_{m,n=1}^{\infty}\beta_m|(\sigma e_m,e_n)|^2
\\&=&\sum_{n=1}^{\infty} (\sigma Q^{1/2} e_n, \sigma Q^{1/2} e_n)=
\sum_{n=1}^{\infty} (Q^{1/2} \sigma^{*}\sigma Q^{1/2} e_n, e_n)\\
&=& \tr((\sigma Q^{1/2})^{*} (\sigma Q^{1/2})) = \tr((\sigma
Q^{1/2}) (\sigma Q^{1/2})^{*}) = \tr(\sigma Q \sigma^{*}).
\end{eqnarray*}
Here we have used the property that, since $\sigma Q^{1/2}$ is a
Hilbert-Schmidt operator,  $\tr((\sigma Q^{1/2})^{*} (\sigma
Q^{1/2})) = \tr((\sigma Q^{1/2}) (\sigma Q^{1/2})^{*})$.
\end{remark}
The next lemma shows that the sum of linear and non-linear operator
is locally monotone in the $\textrm{L}^4$-ball.

\begin{lemma}\label{Mon}
For a given $r>0$, let $B_r$ denote the $\textrm{L}^4(G)$ ball in
$V$, i.e., $B_r = \{\y \in V : \|\y\|_{\textrm{L}^4(G)} \leq r\}$.
Define the non-linear operator $F$ on $V$ by $F(\x) = -\a\x -
\b(\x).$ Then the pair $\Big(F, \sigma +
\int_Zg(\cdot,z)\lambda(dz)\Big)$ is monotone in $B_r$, i.e., for
any $\x \in V$ and $\y \in B_r$, if $\w$ denotes $\x - \y$,
\begin{eqnarray}\label{monotone}
&&\big(F(\x) - F(\y) ,\w \big) - Cr^4|\w|^2 + |\sigma(t,\x) -\sigma(t,\y)|^2_{\textrm{L}_Q} \nonumber \\
&&\qquad\qquad\qquad\qquad\qquad\quad\quad+ \int_Z|g(\x,z) -
g(\y,z)|^2\lambda(dz) \leq 0.
\end{eqnarray}
\end{lemma}

\begin{proof}
First, it is clear that $\left(\a\w, \w\right) = \|\w\|^2,$ as
\begin{eqnarray*}
\left(\a\w, \w\right) = a((\w,\w)) = [[v,v]]_{V_1} + S[[B,B]]_{V_2}
=[[\w,\w]]_V =\|\w\|^2.
\end{eqnarray*}
Using the bilinearity of the operator $\b$, we get
\begin{eqnarray*}
&&\big(\b(\x,\w),\y\big) = b(\x,\w,\y) =-b(\x,\y,\w)
=-b(\x,\y,\w)-b(\x,\w,\w) \\
&&\qquad\qquad\qquad\;=-b(\x,\y + \w,\w)=-b(\x,\x,\w)
=-\big(\b(\x),\w\big).
\end{eqnarray*}
Hence, $\left(\b(\x),\w\right) = -\left(\b(\x,\w),\y\right).$ Also,
we have,
$$\left(\b(\y,\w),\x\right)=b(\y,\w,\x)=-b(\y,\y,\x)=-b(\y,\y,\w)=-\left(\b(\y),\w\right).$$
Using the two equations above, one obtains
\begin{eqnarray*}
\big(\b(\x) - \b(\y), \w\big)
&=&- b(\x,\w,\y)+b(\y,\w,\x) \\
&=&b(\y,\w,\w) - b(\w,\w,\y)=-\big(\b(\w),\y\big).
\end{eqnarray*}
Using H\"{o}lder's inequality, Young's inequality and Sobolev
embedding theorem, we have for any $\e>0$,
\begin{eqnarray*}
\left|\big(\b(\x) - \b(\y),\w\big)\right| &=& \left|-\big(\b(\w),\y\big)\right|= \left|b(\w,\w,\y)\right|\\
&\leq& \|\w\|_{\textrm{L}^4(G)}\|\w\|\|\y\|_{\textrm{L}^4(G)}\\
&\leq& \|\w\|^{3/2}|\w|^{1/2}\|\y\|_{\textrm{L}^4(G)}\\
&\leq& \varepsilon \|\w\|^2 +
C_\varepsilon|\w|^2\|\y\|^4_{\textrm{L}^4(G)}.
\end{eqnarray*}
Here $C_\varepsilon>0$ is a constant that depends on $\varepsilon$
and denote it by $C$. Using the definition of the operator $F$
yields
\begin{eqnarray}\label{equation3}
\big(F(\x) - F(\y), \w\big)
&=&-\big(\a\w,\w\big) - \big(\b(\x) - \b(\y),\w\big)\nonumber\\
&\leq& (-1 + \varepsilon)\|\w\|^2 + C|\w|^2r^4.
\end{eqnarray}
But $V\subset H $ implies $(1-\varepsilon)|\w|^2\leq
(1-\varepsilon)\|\w\|^2.$ Therefore one can have,
$$\big(F(\x) - F(\y), \w\big) +(1 - \varepsilon)|\w|^2 - C|\w|^2r^4\leq 0.$$
But from the Hypothesis \ref{hyp}, $L<1$, choose $\e$ such that $\e
< (1-L)$, one obtains,
$$\big(F(\x) - F(\y), \w\big) +L|\w|^2 - C|\w|^2r^4\leq 0.$$
Using condition $(H.3)$, we get (\ref{monotone}).
\end{proof}

\subsection{Energy estimate and existence theory}
In this section we will find suitable energy estimates and prove the
existence and uniqueness of stong solution of the MHD system. For
that we first define the Galerkin approximations of the MHD system.
Let $H_n := \text{span}\ \{e_1, e_2, \cdots, e_n\}$ where $\{e_j\}$
is any fixed orthonormal basis in $H$ with each $e_j \in
\mathfrak{D}(\a)$. Let $P_n$ denote the orthogonal projection of $H$
to $H_n$. Define $\x_n = P_n \x$.  Let $W_n = P_nW$ . Let $\sigma_n
= P_n\sigma $ and $\int_Zg^n(\cdot,z)\tilde{N}(dt,dz) =
P_n\int_Zg(\cdot,z)\tilde{N}(dt,dz)$, where $g^n=P_ng$. Define
$\x_n$ as the solution of the following stochastic differential
equation in the variational form such that for each $\v \in H_n$ and
with $\x_n(0) = P_n \x(0)$,
\begin{eqnarray}\label{variational}
&&\d (\x_n(t) , \v) = (F(\x_n(t)), \v)\d t + (\sigma_n(t, \x_n(t)) \d W_n(t), \v) \nonumber \\
&&\qquad\qquad\qquad\qquad\qquad\qquad+
\int_Z\big(g^n(\x_n(t-),z),\v\big)\tilde{N}(dt,dz).
\end{eqnarray}

\begin{theorem}\label{energy}
Under the above mathematical setting, let  $\x(0)$ be
$\mathscr{F}_0$ - measurable, $\sigma_n \in C([0, T] \times V;
L_Q(H_0; H))$,  $g^n\in \mathbb{H}^2_{\lambda}([0, T] \times Z; H)$
and let $\mathbb{E}|\x(0)|^2 < \infty$. Let $\x_n(t)$ denote the
unique strong solution to finite system of equations
\eqref{variational} in $\mathscr{D}([0, T], H_n)$. Then with $K$ as
in condition (H.2), the following estimates hold:

\noindent For all $0 \leq t \leq T$,
\begin{eqnarray}
\mathbb{E}| \x_n(t)|^2 + 2\int_0^t \mathbb{E}\| \x_n(s)\|^2 \d s\leq
\left(1 + KTe^{KT}\right)\left(\mathbb{E}|\x(0)|^2 + KT
\right)\textrm{ and} \label{energy1}
\end{eqnarray}
\begin{eqnarray}
\mathbb{E}\left[\sup_{0\leq t\leq T} |\x_n(t)|^2 \right]+ 4\int_0^T
\mathbb{E}\|\x_n(t)\|^2 \d t \leq C\left(\mathbb{E}|\x(0)|^2,K,
T\right)\label{energy2}.
\end{eqnarray}

\noindent Also for any $K > \delta > 0$,
\begin{eqnarray}\label{energy3}
\mathbb{E}|\x_n(t)|^2 e^{-\delta t} + 2\int_0^t
\mathbb{E}\|\x_n(s)\|^2 e^{-\delta s} \d s \leq
C\left(\mathbb{E}|\x(0)|^2, K, \delta, T\right) \textrm{ and}
\end{eqnarray}
\begin{eqnarray}\label{energy4}
\mathbb{E}\left[\sup_{0\leq t\leq T} |\x_n(t)|^2e^{-\delta t}\right]
+ 4\int_0^T\mathbb{E}\|\x_n(t)\|^2e^{-\delta t}dt \leq
C\left(\mathbb{E}|\x(0)|^2, K, \delta, T\right).
\end{eqnarray}
\end{theorem}

We need the following assumption to get the $p$-th moment estimate
for the stochastic MHD system with L\'{e}vy noise.
\begin{assumption}\label{assum}
Let $p\geq 2$. Then, for all $t \in (0, T)$, there exists a positive
constant $K_1$ such that for all $\x \in H$,
\begin{eqnarray}\label{eqtn99}
|\sigma(t,\x(t))|_{\mathrm{L}_Q}^p+\int_Z|g(\x(t),z)|_H^p\lambda(dz)\leq
K_1\left(1+|\x(t)|^p\right).
\end{eqnarray}
\end{assumption}

\begin{theorem}\label{thm99}
Under the above mathematical setting, let  $p\geq 2$, $\x(0)$ be
$\mathscr{F}_0$ - measurable, $\sigma_n \in C([0, T] \times V;
L_Q(H_0; H))$,  $g^n\in \mathbb{H}^p_{\lambda}([0, T] \times Z; H)$
and let $\mathbb{E}|\x(0)|^p< \infty$. Let $\x_n(t)$ denote the
unique strong solution to finite system of equations
(\ref{variational}) in $\mathscr{D}([0, T], H_n)$. Then with $K$ as
in condition (H.2) and $K_1$ as in Assumption \ref{assum}, the
following estimates hold:
\begin{align}\label{90}
\mathbb{E}(\sup_{0\leq t\leq T} |\x_n(t)|^{p})
+p\mathbb{E}\int_0^T\|\x_n(t)\|^2|\x_n(t)|^{p-2}dt \leq
C\left(\mathbb{E}|\x(0)|^p,K,K_1,p, T\right).
\end{align}
\end{theorem}
\begin{proof}
Define $\tau_N = \inf\left\{t: |\x_n(t)|^p + \int_0^t\|\x_n(s)\|^p
\d s > N\right\}$ as the stopping time. We have to find the $p^{th}$
moment estimates for the above system (for similar formulation see
Theorem 4.4 of \cite{AM}). For this, let us take the function
$f(x)=|x|^p$ and apply the It\^{o}'s formula (see Theorem 5.1,
chapter II of \cite{IW}, Theorem 4.4.7 of \cite{Ap}, Theorem 4.4 of \cite{RuZ}) to the process $\x_n(t)$ to
obtain,
\begin{eqnarray}\label{eqtn102}
&&|\x_n(\t)|^p+p\int_0^{\t}\|\x_n(s)\|^2|\x_n(s)|^{p-2}\d
s\nonumber\\&&=|\x(0)|^p+p\int_0^{\t}|\x_n(s)|^{p-2}\left(\sigma_n(s,\x_n(s))\d
W(s),\x_n(s)\right)\nonumber\\&&+\frac{p(p-1)}{2}\int_0^{\t}|\x_n(s)|^{p-2}\textrm{Tr}(\sigma_n(s,\x_n(s))
Q\sigma^*_n(s,\x_n(s)))\d s+I,
\end{eqnarray}
\begin{eqnarray*}
I&=&\int_0^{\t}\int_Z\Big[p|\x_n(s-)|^{p-2}\left(g^n(\x_n(s-),z),\x_n(s-)\right)\Big]\tilde{N}(ds,dz)\\
&&\qquad+\int_0^{\t}\int_Z\Big[|\x_n(s-)+g^n(\x_n(s-),z)|^p-|\x_n(s-)|^p\\
&&\qquad\qquad\qquad\qquad-p|\x_n(s-)|^{p-2}\left(g^n(\x_n(s-),z),\x_n(s-)\right)\Big]N(\d
s,\d z).
\end{eqnarray*}
We take the term
$\frac{p(p-1)}{2}\int_0^{\t}|\x_n(s)|^{p-2}\textrm{Tr}(\sigma_n(s,\x_n(s))
Q\sigma^*_n(s,\x_n(s)))\d s$ from (\ref{eqtn102}) and apply Young's
inequality to obtain,
\begin{align}\label{eqtn103}
&\frac{p(p-1)}{2}\int_0^{\t}|\x_n(s)|^{p-2}\textrm{Tr}(\sigma_n(s,\x_n(s))Q\sigma^*_n(s,\x_n(s)))\d
s\nonumber\\&\leq
\frac{p}{2}\int_0^{\t}\|\x_n(s)\|^2|\x_n(s)|^{p-2}\d
s+C_1(p)\int_0^{\t}|\sigma_n(s,\x_n(s))|^p\d s,
\end{align}
where
$C_1(p)=(p-1)^{p/2}\left(\frac{p-2}{p}\right)^{\frac{p-2}{2}}$. Now
applying (\ref{eqtn103}) in (\ref{eqtn102}) to obtain,
\begin{eqnarray}\label{eqtn104}
&&|\x_n(\t)|^p+\frac{p}{2}\int_0^{\t}\|\x_n(s)\|^2|\x_n(s)|^{p-2}\d
s\nonumber\\&&\qquad\leq
|\x(0)|^p+p\int_0^{\t}|\x_n(s)|^{p-2}\left(\sigma_n(s,\x_n(s))\d
W(s),\x_n(s)\right)
\nonumber\\&&\qquad\qquad+C_1(p)\int_0^{\t}|\sigma_n(s,\x_n(s))|^p\d
s+I.
\end{eqnarray}
Let us take expectation on both sides of the inequality
(\ref{eqtn104}). By using Hypothesis (H.2) (linear growth property),
we have $\mathbb{E}\int_0^{\t}
|\x_n(s)|^2|\sigma_n(t,\x_n(t))|^2_{\mathrm{L}_Q}\d s<\infty$ and
$\mathbb{E}\int_0^{\t}\int_Z|\x_n(t-)|^2|g^n(\x_n(t-),z)|^2\lambda(dz)\d
s<\infty$. Then, the processes
$\int_0^{\t}|\x_n(s)|^{p-2}\left(\sigma_n(s,\x_n(s))\d
W(s),\x_n\right)$ and
$\int_0^{\t}\int_Z\Big[|\x_n(s)+g^n(\x_n(s),z)|^p$
$-|\x_n(s)|^p\Big]\tilde{N}(ds,dz)$ are martingales having zero
averages (see Proposition 4.10 of R\"{u}diger \cite{Ru}). The second
term in $I$ can be estimated by the following Taylor's formula
(there exits a constant $C_p\geq 0(p\geq 2)$),
$$\left||\x_n+h|^p-|\x_n|^p-p|\x_n|^p(\x_n,h)\right|\leq C_p(|\x_n|^{p-2}|h|^2+|h|^p),\;\;\forall\x_n,h\in H_n)$$ and Hypothesis
\ref{hyp}. Then on taking $N\rightarrow\infty$, $\t\rightarrow t$,
one can show that,
$$\mathbb{E}\left[|\x_n(t)|^p
+p\int_0^t\|\x_n(s)\|^2|\x_n(s)|^{p-2} \d s \right]\leq
C\left(\mathbb{E}|\x(0)|^p,K,K_1,p, T\right).$$ Let us take the
supremum from $0\leq t\leq \T$ and then taking the expectation on
the inequality (\ref{eqtn104}), one gets,
\begin{eqnarray}\label{eqtn105}
&&\mathbb{E}\left[\sup_{0\leq t\leq
\T}|\x_n(t)|^p+\frac{p}{2}\int_0^{\T}\|\x_n(t)\|^2|\x_n(t)|^{p-2}\d
t\right]\nonumber\\&&\leq\mathbb{E}\left[|\x(0)|^p\right]+p\mathbb{E}\left[\sup_{0\leq
t\leq \T}\left|\int_0^t|\x_n(s)|^{p-2}\left(\sigma_n(s,\x_n(s))\d
W(s),\x_n(s)\right)\right|\right]\nonumber\\&&\quad+C_1(p)
\mathbb{E}\left(\int_0^{\T}|\sigma_n(s,\x_n(s))|^p\d
s\right)+\mathbb{E}\left[\sup_{0\leq t\leq \T}I\right].
\end{eqnarray}
Take the term $p\mathbb{E}\left[\sup_{0\leq t\leq
\T}\left|\int_0^t|\x_n(s)|^{p-2}\left(\sigma_n(s,\x_n(s))\d
W(s),\x_n(s)\right)\right|\right]$ from the inequality
(\ref{eqtn105}) and apply Burkholder-Davis-Gundy inequality and
Young's inequality to get,
\begin{eqnarray}\label{eqtn106}
&&p\mathbb{E}\left[\sup_{0\leq t\leq
\T}\left|\int_0^t|\x_n(s)|^{p-2}\left(\sigma_n(s,\x_n(s))\d
W(s),\x_n(s)\right)\right|\right]\nonumber\\&&\leq p\mathbb{E}\left[\sup_{0\leq t\leq
\T}|\x_n(t)|^{p-1}\left(\int_0^{\T}|\sigma_n(s,\x_n(s))|^2\d
s\right)^{1/2}\right]\\&&\leq \frac{1}{2}\mathbb{E}\left[\sup_{0\leq
t\leq
\T}|\x_n(t)|^p\right]+(2(p-1))^{p-1}\mathbb{E}\left(\int_0^{\T}|\sigma_n(s,\x_n(s))|^2\d
s\right)^{p/2}.\nonumber
\end{eqnarray}
Let us now take the term
$\mathbb{E}\left(\int_0^{\T}|\sigma_n(s,\x_n(s))|^2\d
s\right)^{p/2}$ from the inequality (\ref{eqtn106}) and apply
H\"{o}lder's inequality with $\frac{p}{p-2}$ and $\frac{p}{2}$ as
the exponents to get,
\begin{align}\label{eqtn107}
\mathbb{E}\left(\int_0^{\T}|\sigma_n(s,\x_n(s))|^2\d
s\right)^{p/2}\leq
T^{\frac{p-2}{2}}\mathbb{E}\left(\int_0^{\T}|\sigma_n(s,\x_n(s))|^p\d
s\right).
\end{align}
Use (\ref{eqtn107}) in (\ref{eqtn106}), then (\ref{eqtn105}) reduces
to,
\begin{align}\label{eqtn109}
&\frac{1}{2}\mathbb{E}\left[\sup_{0\leq t\leq
\T}|\x_n(t)|^p\right]+\frac{p}{2}\mathbb{E}\left(\int_0^{\T}\|\x_n(s)\|^2|\x_n(s)|^{p-2}\d
s\right)\nonumber\\&\leq
\mathbb{E}|\x(0)|^p+C_1(p,T)\mathbb{E}\left(\int_0^{\T}|\sigma_n(s,\x_n(s))|^p\d
s\right)+\mathbb{E}\left(\sup_{0\leq t\leq \T}I\right),
\end{align}
where
$C_1(p,T)=\left[(2(p-1))^{p-1}T^{\frac{p-2}{2}}+C_1(p)\right]$. Let
us take stochastic integral as $\x_n(t)=\int_0^t\int_Z
g^n(\x_n(s-),z)\tilde{N}(ds,dz)$ and apply It\^{o}'s lemma to the
function $|x|^p$ for the process $\x_n(t)$. Then apply
Kunita's inequality (see Lemma \ref{lemma100}, for more details see
Theorem 4.4.23 and Corollary 4.4.24 of \cite{Ap}), we have,
\begin{align}\label{eqtn110}
\mathbb{E}\left[\sup_{0\leq t\leq \T}|I(t)|\right]&\leq
C_2(p,T)\left\{\mathbb{E}\left[\int_0^{\T}\left(\int_Z|g^n(\x_n(s-),z)|^2\lambda(dz)\right)^{p/2}\d
s\right]\right.\nonumber\\&\left.\quad+\mathbb{E}\left[\int_0^{\T}\int_Z|g^n(\x_n(s-),z)|^p\lambda(dz)\d
s\right]\right\}.
\end{align}
Apply Assumption \ref{assum} on the term
$C_1(p,T)\mathbb{E}\left(\int_0^{\T}|\sigma_n(s,\x_n(s))|^p\d
s\right)$, then
\begin{eqnarray}\label{eqtn112}
&&C_1(p,T)\mathbb{E}\int_0^{\T}|\sigma_n(s,\x_n(s))|^p\d
s\nonumber\\&&\leq
C_1(K_1,p,T)T+C_1(K_1,p,T)\mathbb{E}\int_0^{\T}\sup_{0\leq
s\leq t}|\x_n(s)|^p\d s.
\end{eqnarray}
In the inequality (\ref{eqtn110}), let us denote the RHS as $I_1$.
Apply Hypothesis \ref{hyp} and the Assumption \ref{assum} and using
the estimate $(a+b)^p\leq 2^{p-1}(a^p+b^p)$ to obtain,
\begin{align}\label{eqtn113}
|I_1|&\leq
C_2(p,T)\left\{\mathbb{E}\left[\int_0^{\T}\left(K(1+|\x_n(s)|^2)\right)^{p/2}\d
s\right]\right.\nonumber\\&\left.\qquad+\mathbb{E}\left[\int_0^{\T}K_1\left(1+|\x_n(s)|^p\right)\d
s\right]\right\}\nonumber\\&\;\leq
C_3(K,K_1,p,T)T+C_3(K,K_1,p,T)\mathbb{E}\int_0^{\T}\sup_{0\leq s\leq
t}|\x_n(s)|^p\d t,
\end{align}
where $C_3(K,K_1,p,T)=\left(C_2(p,T)K^{p/2}2^{p/2-1}+K_1\right).$
Hence we have,
\begin{align}
&\mathbb{E}\left[\sup_{0\leq t\leq
\T}|\x_n(t)|^p\right]+p\mathbb{E}\left(\int_0^{\T}\|\x_n(s)\|^2|\x_n(s)|^{p-2}\d
s\right)\label{eqtn115}\\&\leq
2\mathbb{E}|\x(0)|^p+C_4(K,K_1,p,T)T+C_4(K,K_1,p,T)\mathbb{E}\left(\int_0^{\T}\sup_{0\leq
s\leq t}|\x_n(s)|^p\d t\right),\nonumber
\end{align}
where $C_4(K,K_1,p,T)=2\left[C_1(K_1,p,T)+C_3(K,K_1,p,T)\right].$ In
particular, we have,
\begin{eqnarray*}
\mathbb{E}\left[\sup_{0\leq t\leq \T}|\x_n(t)|^p\right]&\leq&
2\mathbb{E}|\x(0)|^p+C_4(K,K_1,p,T)T\nonumber\\&&+C_4(K,K_1,p,T)\mathbb{E}\int_0^{\T}\sup_{0\leq
s\leq t}|\x_n(s)|^p\d t.
\end{eqnarray*}
Applying Gronwall's inequality to get,
\begin{eqnarray}\label{eqtn116}
\mathbb{E}\left[\sup_{0\leq t\leq \T}|\x_n(t)|^p\right]\leq
\left(2\mathbb{E}|\x(0)|^p+C_4(K,K_1,p,T)T\right)e^{C_4(K,K_1,p,T)T}.
\end{eqnarray}
By applying (\ref{eqtn116}) in (\ref{eqtn115}) and taking
$N\rightarrow\infty$, $\T\rightarrow T$, we get (\ref{90}).
\end{proof}

\begin{theorem}\label{thm100}
Under the above mathematical setting, let $p\geq 2$, $\x(0)$ be
$\mathscr{F}_0$ - measurable, $\sigma_n \in C([0, T] \times V;
L_Q(H_0; H))$,  $g^n\in \mathbb{H}^p_{\lambda}([0, T] \times Z; H)$
and let $\mathbb{E}|\x(0)|^p< \infty$. Let $\x_n(t)$ denote the
unique strong solution to finite system of equations
(\ref{variational}) in $\mathscr{D}([0, T], H_n)$.
Then with $K$ as in condition (H.2) and $K_1$ as in Assumption \ref{assum}, the following estimate hold
for given $\delta > 0$ and for all $0\leq t\leq T$,
$$
\mathbb{E}\sup_{0\leq t\leq T} |\x_n(t)|^pe^{-\delta t}
+(p+2\delta)\mathbb{E}\int_0^T|\x_n(t)|^pe^{-\delta t}dt \leq
C(\mathbb{E}|\x(0)|^2,K,K_1, \delta,p, T).
$$
\end{theorem}
\begin{proof}
The proof the Theorem can be obtained by applying It\^{o}'s formula
to the function $e^{-\delta t}|x|^p$ by taking the process
$x=\x_n(t)$ and following Theorem \ref{thm99}.
\end{proof}

\begin{definition}($Strong\ Solution$)
A strong solution $\x$ is defined on a given probability space
$(\Omega, \mathscr{F}, \mathscr{F}_{t}, \mathbb{P})$ as a
$\mathrm{L}^2(\Omega;\mathrm{L}^{\infty }(0,T; H)\cap
\mathrm{L}^2(0,T; V)\cap \mathscr{D}(0,T; H))$ valued adapted
process which satisfies the stochastic MHD system
\begin{eqnarray}\label{15}
\d \x + \big[\a\x + \b(\x)\big] \d t &= &\sigma(t, \x) \d W(t) + \int_Zg(\x,z)\tilde{N}(dt,dz)\\
\x(0) &=& \x_0 \in H,\nonumber
\end{eqnarray}
 in the
weak sense and also the energy inequalities in Theorem \ref{energy}.
\end{definition}
Monotonicity arguments were first used by Krylov and
Rozovskii\cite{KR} to prove the existence and uniqueness of the
strong solutions for a wide class of stochastic evolution equations
(under certain assumptions on the drift and diffusion coefficients),
which in fact is the refinement of the previous results by
Pardoux\cite{Pa1, Pa2} (also see Mativier\cite{Me}) and also the
generalization of the results by Bensoussan and Temam\cite{Be}.
Menaldi and Sritharan\cite{Ms} further developed this theory for the
case when the sum of the linear and nonlinear operators are locally
monotone. Since L\'{e}vy noise appears in this paper, the proof of
existence and uniqueness is given in complete form.


\begin{theorem}\label{existence}
Let  $\x(0)$ be $\mathscr{F}_0$ measurable and $\ \mathbb{E}|\x_0|^2
< \infty.$ The diffusion coefficient satisfies the conditions
(H.1)-(H.3). Then there exists a unique strong solution $\x(t,x, w)$
with the regularity
$$\x \in \mathrm{L}^2(\Omega;\mathrm{L}^{\infty }(0,T; H)\cap
\mathrm{L}^2(0,T; V)\cap \mathscr{D}(0,T; H))$$ satisfying the
stochastic MHD system \eqref{15} and a priori bounds
in Theorem \ref{energy}.
\end{theorem}

\begin{proof}
\textbf{Part I (Existence).}
 Using the a priori estimate in the Theorem
\ref{energy}, it follows from the Banach-Alaoglu theorem that along
a subsequence, the Galerkin approximations $\{\x_n\}$ have the
following limits:
\begin{eqnarray}
&&\x_n\longrightarrow \x\quad  \text {weak star in}\
\mathrm{L}^2(\Omega ;
  \mathrm{L}^{\infty}(0,T; H)) \cap\mathrm{L}^{2}
  (\Omega;\mathrm{L}^{2}(0,T; V)),\nonumber\\
&& F(\x_n)\longrightarrow F_0\quad \text{weakly in}\
\mathrm{L}^{2}(\Omega;\mathrm{L}^{2}(0,T; V^{\prime})),\nonumber\\
&& \sigma_n(\cdot, \x_n) \longrightarrow S\quad \text{weakly in}\ \mathrm{L}^{2}(\Omega;\mathrm{L}^{2}(0,T; \mathrm{L}_Q)),\nonumber \\
&&g(\x_n,\cdot)\longrightarrow G\quad \text{weakly in}\
\mathbb{H}^2_{\lambda}([0, T] \times Z; H).\label{16}
\end{eqnarray}
The assertion of the second statement holds since $F(\x_n)$ is
bounded in \\ $\mathrm{L}^{2}(\Omega;\mathrm{L}^{2}(0,T;
V^{\prime}))$.
 Likewise since diffusion coefficient has the linear growth property and $\x_n$ is bounded in $\mathrm{L}^2(0, T; V)$ uniformly in $n$,
 the last two statements hold. Then $\x$ satisfies the It\^{o} differential
\begin{align*}
\d \x(t) = F_0(t)\d t + S(t)\d W(t)+\int_ZG(t)\tilde{N}(dt,dz)
\text{ weakly in }\
\mathrm{L}^{2}(\Omega;\mathrm{L}^{2}(0,T;V^\prime)).
\end{align*}
Then as Theorem \ref{energy}, one can prove that, $
\mathbb{E}\left[\sup_{0\leq t\leq T}|\x(t)|^2+\int_0^T\|\x(t)\|^2\d
t\right]\leq C\left(\mathbb{E}|\x(0)|^2, K, T\right).
 $
 Then by Gy\"{o}ngy and Krylov \cite{GK}, we have $\x$ is an
 $H$-valued c\'{a}dl\'{a}g $\mathscr{F}_t$-adapted process
 satisfying
 \begin{eqnarray*}
 &&|\x(t)|^2=|\x(0)|^2+2\int_0^t(F_0(s),\x(s))\d s+2\int_0^t(S(s)\d
 W(s),\x(s))\\&&+\int_0^t|S|^2_{\mathrm{L}_Q}\d
 s+2\int_0^t\int_Z(G(s),\x(s))\tilde{N}(d s,d
 z)+\int_0^t\int_Z|G(s)|^2N(ds,dz).
 \end{eqnarray*}
 Let us set,
$ r(t):= 2C\int_0^t \|\v(s)\|_{\mathrm{L}^4}^4 \d s, $ where
$\v(\omega, t, x)$ is any adapted process in
$\mathrm{L}^{\infty}(\Omega \times (0, T); H)$. Apply It\^{o}'s
Lemma to the function $2e^{-r(t)} |x|^2$ and to the process
$\x_n(t)$, integrating from $0\leq t\leq T$ and taking expectation,
one obtains,
\begin{eqnarray*}
&&\mathbb{E}\left[e^{-r(T)} |\x_n(T)|^2 - |\x_n(0)|^2\right]\\
&&\quad= \mathbb{E}\left[\int_0^T e^{-r(t)}\big(2F(\x_n(t)) - \dot{r}(t) \x_n(t), \x_n(t)\big)\d t\right]\\
&&\quad\quad + \mathbb{E}\int_0^T e^{-r(t)} |\sigma_n(t,
\x_n(t))|^2_{\mathrm{L}_Q}\d t
+ \mathbb{E}\int_0^T e^{-r(t)}\int_Z\big|g^n(\x_n(t),z)\big|^2\lambda(dz)\d t\\
&&\quad\quad + 2 \mathbb{E}\int_0^T e^{-r(t)}\big(\sigma_n(t,
\x_n(t))\d W(t), \x_n(t)\big) \\ &&\quad\quad +
2\mathbb{E}\int_0^Te^{-r(t)}\int_Z\big(\x_n(t-),g^n(\x_n(t-),z)\big)\tilde{N}(dt,dz).
\end{eqnarray*}
The last two terms of the above inequality are martingales having
zero averages. Then by the lower semi-continuity property of the
$\mathrm{L}^2$-norm and strong convergence of the initial data, we
obtain,
\begin{eqnarray}
&&\liminf_n \mathbb{E}\left[\int_0^T e^{-r(t)}\big(2F(\x_n(t)) - \dot{r}(t) \x_n(t), \x_n(t)\big)\d t \nonumber \right.\\
&&\left.\quad + \int_0^T e^{-r(t)} |\sigma_n(t,
\x_n(t))|^2_{\mathrm{L}_Q}\d t
+ \int_0^T e^{-r(t)}\int_Z\big|g^n(\x_n(t),z)\big|^2\lambda(dz)\d t\right]\nonumber\\
&&\quad\quad = \liminf_n \mathbb{E}\left[e^{-r(T)} |\x_n(T)|^2 - |\x_n(0)|^2\right]\geq \mathbb{E}\left[e^{-r(T)} |\x(T)|^2 - |\x(0)|^2\right]\nonumber\\
&&\quad\quad = \mathbb{E}\left[\int_0^T e^{-r(t)}\big(2F_0(t) -
\dot{r}(t) \x(t), \x(t)\big)\d t + \int_0^T e^{-r(t)}
|S|^2_{\mathrm{L}_Q}\d t\nonumber\right.\\
&&\left.\quad\quad\quad\quad\quad\qquad\qquad +
\int_0^Te^{-r(t)}\int_Z|G|^2\lambda(dz)\d t\right].\label{lsc}
\end{eqnarray}
Now by monotonicity property from Lemma \ref{Mon} (by choosing
$\x=\x_n(t)$ and $\y=\v(t)$ in (\ref{monotone})), we have,
\begin{eqnarray*}
&&2\mathbb{E}\left[\int_0^T e^{-r(t)} \big(F(\x_n(t)) - F(\v(t)),
\x_n(t) - \v(t)\big)\d t\right]\\
&&\quad - \mathbb{E}\left[\int_0^T e^{-r(t)} \dot{r}(t)|\x_n(t) - \v(t)|^2 \d t\right]  \\
&&\quad +  \mathbb{E}\left[\int_0^T e^{-r(t)} |\sigma_n(t, \x_n(t)) - \sigma_n(t, \v(t))|^2_{\mathrm{L}_Q}\d t\right]\\
&&\quad +
\mathbb{E}\left[\int_0^Te^{-r(t)}\int_Z|g^n(\x_n(t),z)-g^n(\v(t),z)|^2\lambda(dz)\d
t\right] \leq 0.
\end{eqnarray*}
On rearranging the terms, we get,
\begin{eqnarray*}
&&\mathbb{E}\left[\int_0^T e^{-r(t)}\big(2F(\x_n(t)) - \dot{r}(t) \x_n(t), \x_n(t)\big)\d t \right.\\
&&\left.\quad +\int_0^T e^{-r(t)}  |\sigma_n(t,
\x_n(t))|^2_{\mathrm{L}_Q}\d t\ +
\int_0^Te^{-r(t)}\int_Z|g(\x_n(t),z)|^2\lambda(dz)\d t\right]\\
&&\quad\quad\leq \mathbb{E}\left[\int_0^T e^{-r(t)}\big(2F(\x_n(t))-\dot{r}(t)(2\x_n(t) - \v(t)), \v(t)\big) \d t\right] \\
&&\quad\quad\quad + \mathbb{E}\left[\int_0^T e^{-r(t)}\big(2F(\v(t)), \x_n(t)-\v(t)\big) \d t\right] \\
&&\quad\quad\quad +   \mathbb{E}\left[\int_0^T e^{-r(t)}\big(2\sigma_n(t, \x_n(t))-\sigma_n(t, \v(t)), \sigma_n(t, \v(t))\big)_{\mathrm{L}_Q}\d t\right] \\
&&\quad\quad\quad +
\mathbb{E}\left[\int_0^Te^{-r(t)}\int_Z\big(2g^n(\x_n(t),z)-g^n(\v(t),z),g^n(\v(t),z)\big)\lambda(dz)\d
t\right].
\end{eqnarray*}
Taking limit in $n$, using the result from \eqref{lsc}, we have,
\begin{eqnarray*}
&&\mathbb{E}\left[\int_0^T e^{-r(t)}\big(2F_0(t) - \dot{r}(t) \x(t),
\x(t)\big)\d t +\int_0^T e^{-r(t)} |S|^2_{\mathrm{L}_Q}\d
t\nonumber\right.\\ &&\left.\quad\quad +
\int_0^Te^{-r(t)}\int_Z|G|^2\lambda(dz)\d t\right]\\
&&\quad\quad\ \leq
\mathbb{E}\left[\int_0^Te^{-r(t)}\big(2F_0(t)-\dot{r}(t)(2\x(t) -
\v(t)), \v(t)\big) \d t\right]\\
&&\quad\quad\quad + \mathbb{E}\left[\int_0^T
e^{-r(t)}\big(2F(\v(t)), \x(t)-\v(t)\big) \d t\right]\end{eqnarray*}
\begin{eqnarray*}
&&\quad\quad\quad +   \mathbb{E}\left[\int_0^T
e^{-r(t)}\left(2S(t)-\sigma(t, \v(t)), \sigma(t,
\v(t))\right)_{\mathrm{L}_Q}\d t\right]
\\ &&\quad\quad\quad +
\mathbb{E}\left[\int_0^Te^{-r(t)}\int_Z\left(2G(t)-g(\v(t),z),g(\v(t),z)\right)\lambda(dz)\d
t\right].
\end{eqnarray*}
On rearranging the terms, we obtain,
\begin{eqnarray*}
&&\mathbb{E}\left[\int_0^T e^{-r(t)}\big(2F_0(t)-2F(\v(t)), \x(t)-\v(t)\big)\d t\right] \\
&&\quad + \mathbb{E}\left[\int_0^T e^{-r(t)}\dot{r}(t) |\x(t) - \v(t)|^2 \d t\right]\\
&&\quad +  \mathbb{E}\left[\int_0^T e^{-r(t)} \|\textbf{S}(t)-\sigma(t, \v(t))\|^2_{\mathrm{L}_Q}\d t\right]\\
&&\quad +
\mathbb{E}\left[\int_0^Te^{-r(t)}\int_Z\|\textbf{G}(t)-g(\v(t),z)\|^2\lambda(dz)\d
t\right]\leq 0.
\end{eqnarray*}
This estimate holds for any
$\v\in\mathrm{L}^2(\Omega;\mathrm{L}^{\infty}(0,T;\mathbf{H}_m))$
for any $m\in\mathbb{N}$.  It is clear by a density argument that
the above inequality remains the same for any $\v\in
\mathrm{L}^2(\Omega;\mathrm{L}^{\infty}(0,T;\mathbf{H})\cap
 \mathrm{L}^2(0,T;\mathbf{V}))$.  Indeed, for
any
$\v\in\mathrm{L}^2(\Omega;\mathrm{L}^{\infty}(0,T;\mathbf{H})\cap
 \mathrm{L}^2(0,T;\mathbf{V})),$ there exits a
strongly convergent sequence $\v_m\in
\mathrm{L}^2(\Omega;\mathrm{L}^{\infty}(0,T;\mathbf{H})\cap
\mathrm{L}^2(0,T;\mathbf{V}))$ that satisfies the above inequality.
Let $\v(t) = \x(t)$, then, $\textbf{S}(t)=\sigma(t, \x(t))$ and
$\textbf{G}(t)=g(\x(t),z)$. Take $\v = \x - \mu\y$ with $\mu >0$ and
$\y$ is an adapted process in
$\mathrm{L}^2(\Omega;\mathrm{L}^{\infty }(0,T; H)\cap
\mathrm{L}^2(0,T; V)\cap \mathscr{D}(0,T; H))$. Then,
\begin{align*}
\mu \mathbb{E}(\int_0^T e^{-r(t)}\big(2F_0(t)-2F(\x - \mu\y)(t),
\y(t)\big)\d t + \mu\int_0^T e^{-r(t)}\dot{r}(t)|\y(t)|^2\d t)\leq
0.
\end{align*}
Dividing by $\mu$ on both sides of the above inequality and letting
$\mu\to 0$, one obtains
$$
\mathbb{E}\left[\int_0^T e^{-r(t)}\big(F_0(t)-F(\x(t)), \y(t)\big)\d
t\right] \leq 0.
$$

Since $\y(t)$ is arbitrary, we conclude that $F_0(t)=F(\x(t))$. Thus
the existence of the strong solution of the stochastic
MHD system \eqref{15} has been proved.

\noindent \textbf{Part II (Uniqueness).} If $\y\in
\mathrm{L}^2(\Omega;\mathrm{L}^{\infty }(0,T; H)\cap
\mathrm{L}^2(0,T; V)\cap \mathscr{D}(0,T; H))$ be another solution
of the equation \eqref{15}, then $\w=\x-\y$ solves the stochastic
differential equation in $\mathrm{L}^2(\Omega;\mathrm{L}^2(0, T;
V^\prime))$,
\begin{eqnarray}\label{17}
&&\d \w(t) = \big(F(\x(t)) - F(\y(t))\big)\d t + (\sigma(t, \x(t)) - \sigma(t, \y(t))) \d W(t)\nonumber \\
&&\qquad\qquad\qquad\qquad+\int_Z[g(\x(t-),z)-g(\y(t-),z)]\tilde{N}(dt,dz).
\end{eqnarray}
We denote $\sigma_d = \sigma(t, \x(t)) - \sigma(t, \y(t))$ and $g_d
= g(\x(t-),z)-g(\y(t-),z)$.

Let us apply It\^{o}'s Lemma to the function $2e^{-r(t)} |x|^2$ and
to the process $\w(t)$ and use the local monotonicity of the sum of
the linear and nonlinear operators $\a$ and $\b$, e.g. equation
\eqref{equation3}, one obtains,
\begin{eqnarray*}
&&\d \left[e^{-r(t)}  |\w(t)|^2\right]  + (1-\varepsilon)\|\w(t)\|^2e^{-r(t)}\d t \nonumber\\
&&\qquad \leq  e^{-r(t)} |\sigma_d|^2\d t + 2e^{-r(t)} (\sigma_d\d W(t), \w(t)) \nonumber \\
&&\qquad \quad + e^{-r(t)}\int_Z|g_d|^2N(dt,dz) +
2e^{-r(t)}\int_Z\big(\w(t),g_d\big)\tilde{N}(dt,dz).
\end{eqnarray*}
Now integrating from $0\leq t\leq T$ and taking the expectation on
both sides and noting that $0< \varepsilon < 1-L$. Also using the
fact that $2\int_0^T e^{-r(t)} (\sigma_d\d W(t), \w(t)) $ and
$2\int_0^Te^{-r(t)}\int_Z\big(\w(t),g_d\big)\tilde{N}(dt,dz)$ are
martingales having zero averages, we deduce that,
\begin{eqnarray*}
&&\mathbb{E}\left[e^{-r(t)}|\w(t)|^2\right] + (1-\varepsilon)\mathbb{E}\int_0^Te^{-r(t)}\|\w(t)\|^2\d t  \\
&&\quad \leq \mathbb{E}|\w(0)|^2 + \mathbb{E}\int_0^Te^{-r(t)}
|\sigma_d|^2\d t +
\mathbb{E}\int_0^Te^{-r(t)}\int_Z|g_d|^2\lambda(dz)\d t.
\end{eqnarray*}
Using condition (H.3), we have,
\begin{eqnarray*}
\mathbb{E}\left[e^{-r(t)}|\w(t)|^2\right] +
\left((1-\varepsilon)-L\right)\int_0^Te^{-r(t)}\|\w(t)\|^2\d t
\leq\mathbb{E}|\w(0)|^2.
\end{eqnarray*}
Sine $0<L<(1-\varepsilon)<1$, we obtain $\mathbb{P}$-a.s., $
\mathbb{E}\left[e^{-r(t)}|\w(t)|^2\right] \leq \mathbb{E}|\w(0)|^2,
$ which assures the uniqueness of the strong solution.
\end{proof}

\begin{theorem}\label{thm101}
Let  $\x(0)$ be $\mathscr{F}_0$ - measurable and $\
\mathbb{E}|\x_0|^p < \infty.$ The diffusion coefficient satisfies
the conditions $\sigma \in C([0, T] \times V; L_Q(H_0; H))$, $g\in
\mathbb{H}^p_{\lambda}([0, T] \times Z; H)$, Hypothesis (H.3) and
Assumption \ref{assum}. Then there exists a unique strong solution
$\x(t,x, w)$ with the regularity
$$\x \in \mathrm{L}^p(\Omega;\mathrm{L}^{\infty }(0,T; H)\cap
\mathrm{L}^2(0,T; V)\cap \mathscr{D}(0,T; H))$$ satisfying the
stochastic MHD equation given in (\ref{15}) and the a priori bounds
in Theorem \ref{thm99} and Theorem \ref{thm100}.
\end{theorem}
\begin{proof}
The proof of the Theorem  follows the same steps as in Theorem
\ref{existence} with minor modifications.
\end{proof}

\section{Invariant Measures}
In this section, we consider MHD system with additive L\'{e}vy noise
 given by
\begin{eqnarray}\label{qe1}
\d \x(t) + [\a\x(t) + \b(\x(t))]\d t = \sum_k\sigma_k(t)\d W_k(t) + \int_Zg(t,z)\tilde{N}(dt,dz),
\end{eqnarray}
$\x(0)=\xi$. Here $\sigma(t)=\{\sigma_1(t),\sigma_2(t),\cdots\}$ is
$\ell_2(H)$-valued for all $t\geq 0$ and $W(t)$ is given by
$W(t)=\left\{W_1(t),W_2(t),\cdots\right\}$, where $W_k$ are
independent copies of the standard one-dimensional Wiener process.
Hence $\sigma(t)\d W(t)=\sum_k\sigma_k(t)\d W_k(t)$ is an $H$-valued
noise and the stochastic integral induced by the noise is given by
$\x\mapsto \int_0^T(\sigma(t)\d
W(t),\x):=\sum_k\int_0^T(\sigma_k(t),\x)\d W_k(t).$ Let us assume
that $\sum_k|\sigma_k(t)|^2<\infty$ and
$\int_Z|g(t,z)|^4\lambda(dz)=C_1<\infty$. If $\lambda(\cdot)$ is of
finite measure, then we have $\int_Z|g(t,z)|^2\lambda(dz)\leq
[\lambda(Z)]^{1/2}C_1^{1/2}<\infty$.
\begin{remark}
If $\lambda(\cdot)$ is of $\sigma$-finite measure, then consider a
measurable subset $U_m$ of $Z$ with $U_m\uparrow Z$ and
$\lambda(U_m)<\infty$. Assume that
$\int_{U_m^c}|g(t,z)|^2\lambda(dz)\to 0$ as $m\to\infty$,
$t\in[0,T]$. This condition is automatically satisfied if
$\lambda(Z)<\infty.$ Here also we assume that
$\int_{U_m}|g(t,z)|^4\lambda(dz)=C_2<\infty$. Hence we have,
$\int_{U_m}|g(t,z)|^2\lambda(dz)\leq
[\lambda(U_m)]^{1/2}C_2^{1/2}<\infty.$ For this case, the following
results in this section will follow by replacing $Z$ by $U_m$ (see
Fernando and Sritharan \cite{FS}).
\end{remark}
From the definition of the linear operator (\ref{linear}) and
(\ref{lin}), the eigenvalues of $\a$ depend on the Reynold's numbers
$R_e$ and $R_m$. Let the eigenvalues of $\a$ be denoted by
$0<\lambda_1<\lambda_2\leq\cdots$ and the corresponding eigenvectors
by $e_1,e_2\cdots$, which will form a complete orthonormal system in
$H$. Also by the Poincar\'{e} inequality, we have, $\|\x\|\geq
\lambda_1|\x|^2$.
\begin{lemma}\label{ll1}
Let $B(t)$ be an increasing, progressively measurable process with
$B(0)>0$ a. s. Let $M(t)$ be a c\`{a}dl\`{a}g local martingale with
$M(0)=0$. If $\mathcal{Y}(t)=\int_0^t\frac{1}{B(s)}\d M(s)$
converges a. s. to a finite limit as $t\to\infty$, then
$\lim{t\to\infty}\frac{M^{*}(t)}{B(t)}=0$ a. s. on the set
$\{B(\infty)=\infty\}$ where $M^{*}(t)=\sup_{0\leq s\leq t}|M(s)|.$
\end{lemma}
For the proof of continuous martingale see Lemma 2.1 of Sundar
\cite{Su2}. A simple calculation gives the extension of the result
for the c\`{a}dl\`{a}g local martingale.
\begin{lemma}\label{ll2}
Let $\lim_{t\to\infty}\sum_k|\sigma_k(t)|^2=M_1<\infty$,
$\lim_{t\to\infty}\int_{Z}|g(t,z)|^4\lambda(dz)=M_2<\infty$ and
$\lim_{t\to\infty}\int_{Z}|g(t,z)|^2\lambda(dz)=M_3<\infty$. Then
$\lim_{t\to\infty}\frac{M^{*}(t)}{t}=0$ a. s., 
\begin{eqnarray}
M(t)&=&2\sum_k\int_0^t(\sigma_k(s),\x(s))\d
W_k(s)\nonumber\\&&+\int_0^t\int_Z\left[|g(s-,z)|^2+2(g(s-,z),\x(s-))\right]\tilde{N}(ds,dz).
\end{eqnarray}
\end{lemma}
\begin{proof}
Let us set $B(t)=(\e+t)$ in Lemma \ref{ll1}, for any $\e>0$ and
$\mathcal{Y}(t)=\int_0^t\frac{1}{(\e+s)}\d M(s).$ We have (for more
details about the quadratic variation of L\'{e}vy type stochastic
integrals see Section 4.4.3 of Applebaum \cite{Ap}),
\begin{eqnarray}\label{qe2}
&&[\mathcal{Y},\mathcal{Y}]_t=4\sum_k\int_0^t\frac{(\sigma_k(t),\x(s))^2}{(\e+s)^2}\d
s\nonumber\\&&+\int_0^t\int_Z\frac{\left||g(s-,z)|^2+2(g(s-,z),\x(s-))\right|^2}{(\e+s)^2}N(ds,dz)\nonumber\\
&\leq& 4\sum_k\int_0^t\frac{|\sigma_k(s)|^2|\x(s)|^2}{(\e+s)^2}\d
s+2\int_0^t\int_Z\frac{|g(s-,z)|^4}{(\e+s)^2}N(ds,dz)\nonumber\\&&+4
\int_0^t\int_Z\frac{\left|(g(s-,z),\x(s-))\right|^2}{(\e+s)^2}N(ds,dz)\nonumber\\&=&
4\sum_k\int_0^t\frac{|\sigma_k(s)|^2|\x(s)|^2}{(\e+s)^2}\d s
+\int_0^t\int_Z
\frac{\left[2|g(s,z)|^4+4|g(s,z)|^2|\x(s)|^2\right]}{(\e+s)^2}\lambda(dz)ds
\nonumber\\&&+\int_0^t\int_Z
\frac{\left[2|g(s-,z)|^4+4\left|(g(s-,z),\x(s-))\right|^2\right]}{(\e+s)^2}\tilde{N}(ds,dz)
\end{eqnarray}
Taking expectation on both sides of (\ref{qe2}) and noting that the
last term on the right hand side is a martingale having a zero
average (see Proposition 4.10 of R\"{u}diger \cite{Ru}), we have,
\begin{eqnarray}\label{qe21}
\mathbb{E}[\mathcal{Y},\mathcal{Y}]_t &\leq&
4\sum_k\int_0^t\frac{|\sigma_k(s)|^2\mathbb{E}|\x(s)|^2}{(\e+s)^2}\d
s+2\int_0^t\int_Z\frac{|g(s,z)|^4}{(\e+s)^2}\lambda(dz)ds\nonumber\\&&
+4\int_0^t\int_Z\frac{|g(s,z)|^2\mathbb{E}|\x(s)|^2}{(\e+s)^2}\lambda(dz)ds.
\end{eqnarray}
Let $\lambda_1$ be the first eigenvalue of the operator $\a$. Then
Poincar\'{e} inequality gives $\lambda_1|\w|^2\leq\|\w\|^2.$ Hence
by an application of  It\^{o}'s lemma and Poincar\'{e} inequality in
(\ref{qe1}), we have
$$
\mathbb{E}|\x(t)|^2\leq
\mathbb{E}|\x(0)|^2-2\lambda_1\int_0^t\mathbb{E}|\x(s)|^2\d
s+\int_0^t(\sum_k|\sigma_k(s)|^2
+\int_Z|g(s,z)|^2\lambda(dz))\d s.
$$
Hence by Gronwall's inequality,
$$
\mathbb{E}|\x(t)|^2\leq \mathbb{E}|\x(0)|^2e^{-2\lambda_1
t}+\int_0^te^{-2\lambda_1(t-s)}\left(\sum_k|\sigma_k(s)|^2+\int_Z|g(s,z)|^2\lambda(dz)\right)\d
s.
$$
Letting $t$ tend to $\infty$ in (\ref{qe21}), using the assumptions
in Lemma and the above inequality, we have
\begin{eqnarray}\mathbb{E}[\mathcal{Y},\mathcal{Y}]_{\infty}\leq
C\int_0^{\infty}\frac{1}{(\e+s)^2}\d s<\infty.\end{eqnarray} Hence
$[\mathcal{Y},\mathcal{Y}]_{\infty}<\infty$ a. s. which allows us to
conclude that $\lim_{t\to\infty}\mathcal{Y}(t)$ exists almost
surely. Finally we prove the Lemma by invoking Lemma \ref{ll1}.
\end{proof}

\begin{theorem} (Exponential Stability)\label{Them1}
Let $\x$ and $\y$ be two solutions of (\ref{qe1}) with initial
values $\x_0$ and $\y_0$ respectively. Let us assume that
$\lim_{t\to\infty}\sum_k|\sigma_k(t)|^2=M_1<\infty$ and
$\lim_{t\to\infty}\int_{Z}|g(t,z)|^2\lambda(dz)=M_3<\infty$. If
$\frac{2(M_1+M_3)}{\lambda_1}<1$, where $\lambda_1$ is the first
eigenvalue of the operator $\a$, then
$\lim_{t\to\infty}|\x(t)-\y(t)|=0$ a. s.
\end{theorem}
\begin{remark}
A suitable choice of the Reynold's numbers $R_e$ and $R_m$ (i. e.,
viscosity of the fluid flow and the magnetic field) gives
$\lambda_1$ such that $\frac{2(M_1+M_3)}{\lambda_1}<1$.
\end{remark}
\begin{proof}
Since $\x(t)$ and $\y(t)$ are two solutions of (\ref{qe1}), we have
$$
\x(t)-\y(t)+\int_0^t\a(\x(s)-\y(s))\d
s+\int_0^t(\b(\x(s))-\b(y(s)))\d s=\x_0-\y_0.
$$
By using It\^{o}'s formula, we have,
\begin{eqnarray}\label{qe6}
|\w(t)|^2+2\int_0^t\|\w(s)\|^2\d
s+2\int_0^t(\b(\x(s))-\b(\y(s)),\w(s))\d s=|\w(0)|^2,
\end{eqnarray}
where $\w(t)=\x(t)-\y(t)$ and $\w(0)=\x_0-\y_0$. Using
$$|(\b(\x)-\b(\y),\w)|\leq 2\|\w\||\w|\|\y\|\leq
\frac{1}{2}\|\w\|^2+2|\w|^2\|\y\|^2$$ and Poincar\'{e} inequality
 in (\ref{qe6}), we have,
\begin{eqnarray}
|\w(t)|^2+\lambda_1\int_0^t|\w(s)|^2\d s\leq
|\w(0)|^2+4\int_0^t|\w(s)|^2\|\y(s)\|^2\ d s.
\end{eqnarray}
Hence by applying Gronwall's inequality, we get,
\begin{eqnarray}\label{qe8}
|\w(t)|^2\leq |\w(0)|^2\exp\left(4\int_0^t\|\y(s)\|^2\d
s-\lambda_1t\right).
\end{eqnarray}
By applying It\^{o}'s lemma to the function $|x|^2$ and to the
process $\y(t)$ and using the properties of the linear and bilinear
operators, we have,
\begin{eqnarray}
&&|\y(t)|^2+2\int_0^t\|\y(s)\|^2\d
s=|\y(0)|^2+\sum_k\int_0^t|\sigma_k(s)|^2\d
s\nonumber\\&&\qquad\qquad+2\sum_k\int_0^t(\sigma_k(s),\y(s))\d
W_k(s)+\int_0^t\int_Z|g(s-,z)|^2\lambda(dz)ds\nonumber\\&&\qquad\qquad+\int_0^t\int_Z\left[|g(s-,z)|^2+2(g(s-,z),\y(s-))\right]\tilde{N}(ds,dz).
\end{eqnarray}
Hence from above, we get,
\begin{eqnarray}\label{qe10}
&&\frac{2}{t}\int_0^t\|\y(s)\|^2\d s\leq
\frac{|\y(0)|^2}{t}+\frac{1}{t}\sum_k\int_0^t|\sigma_k(s)|^2\d
s\nonumber\\&&\qquad\qquad+\frac{2}{t}\sum_k\int_0^t(\sigma_k(s),\y(s))\d
W_k(s)+\frac{1}{t}\int_0^t\int_Z|g(s-,z)|^2\lambda(dz)ds\nonumber\\
&&\qquad\qquad+\frac{1}{t}\int_0^t\int_Z\left[|g(s-,z)|^2+2(g(s-,z),\y(s-))\right]\tilde{N}(ds,dz).
\end{eqnarray}
By Lemma \ref{ll2}, we have,
\begin{eqnarray}\label{qe11}
&&\lim_{t\to\infty}\frac{1}{t}\left[2\sum_k\int_0^t(\sigma_k(s),\y(s))\d
W_k(s)\right.\nonumber\\&&\left.\qquad+\int_0^t\int_Z\left[|g(s-,z)|^2+2(g(s-,z),\y(s-))\right]\tilde{N}(ds,dz)\right]=0.
\end{eqnarray}
From the assumption of the theorem, we have,
\begin{eqnarray}\label{qe12}
\lim_{t\to\infty}\frac{1}{t}\sum_k\int_0^t|\sigma_k(s)|^2\d
s=M_1\textrm{ and
}\lim_{t\to\infty}\frac{1}{t}\int_0^t\int_Z|g(s,z)|^2\lambda(dz)ds=M_3.
\end{eqnarray}
By using (\ref{qe11}) and (\ref{qe12}) in (\ref{qe10}), we have,
$
\lim_{t\to\infty}\frac{1}{t}\int_0^t\|\y(s)\|^2\d s\leq
\left(\frac{M_1+M_3}{2}\right).
$
Using this bound in (\ref{qe8}) provided
$\frac{2(M_1+M_3)}{\lambda_1}<1$, we get
$\lim_{t\to\infty}|\x(t)-\y(t)|=0$ a. s.
\end{proof}

\begin{theorem} (Existence and Uniqueness of Invariant Measures)
Let $\sigma(t,x)=\sigma(x)$ and $g(t,x,z)=g(x,z)$. Then under the
hypotheses of the above theorem, there exists a unique stationary
measure with support in $V$, for the solution $\x(t,x,\omega)$ of
the stochastic MHD system with L\'{e}vy noise.
\end{theorem}
\begin{proof}
The method of proving existence of an invariant measure is well
known in the literature and interested readers may look at the works
by Barbu and Da Prato \cite{VB}, Chow and Khasminskii \cite{CK}, Da
Prato and Zabczyk \cite{DaZ,DaZ1}, Da Prato and Debussche
\cite{DaD1}, Flandoli \cite{Fl}, Odasso \cite{Od}, Sundar
\cite{Su2,Su1} for the extensive study on the subject. For the sake
of completeness we are giving a very brief outline of the proof.

Since the semigroup associated with the stochastic MHD system
(\ref{qe1}) has the Feller property, first one proves the time
averages $(\left\{\mu_T\right\}_{T\geq 0})$ of law of the solution
(probability measure) is tight in $H$ by the method developed by
Chow and Khasminskii (see \cite{CK}) from the energy estimates.
 Hence due to Prokhorov's Theorem $\{\mu_T\}_{T\geq
0}$ in $H$ is relatively compact. So there exists a sequence
$\mu_{T_n}$ weakly convergent to a probability measure $\mu$.
Finally one proves $\mu$ is invariant by the classical method of
Krylov and Bogoliubov (Theorem 3.2 (step 4) of Flandoli \cite{Fl}).
One can also prove the following moment estimates of the invariant
measure, $\mu\left(\|\cdot\|^2\right)\leq
\left(\frac{K}{2-K}\right)$ and $\mu[|\cdot|^p]\leq C(K,K_1,p)$.

For uniqueness, let us assume that $\mu_1$ and $\mu_2$ be two
probability measures on $H$ that are stationary for the equation
(\ref{qe1}). We have to show that $\mu_1=\mu_2$. For this we have to
prove for all $\phi\in C_b(H)$,
$\int_H\phi(z)\d\mu_1(z)=\int_H\phi(z)\d\mu_2(z).$ Let $\x^{\xi}$ be
the solution of the equation (\ref{qe1}) with $\x(0)=\xi$. Since
$\mu_1$ and $\mu_2$ are invariant measures, by definition we have,
$\mu_1(B)=\int_HP_{\xi}(t,B)\d \mu_1(\xi)$ and $\mu_2(B)=\int_HP_{\xi}(t,B)\d
\mu_2(\xi),$ where $P_{\xi}(t,B)=\mathbb{P}\left\{\x^{\xi}(t)\in
B\right\}\textrm{ and }\x^{\xi}(0)=\xi.$ Define
$\mu_t^{\xi}(B)=\frac{1}{t}\int_0^tP_{\xi}(s,B)\d s$ for all
$B\in\mathscr{B}(H)$. Also, we have,
$\mathbb{E}\left(\phi(\x^{\xi}(t))\right)=\int_H\phi(x)P_{\xi}(t,\d
x).$ By using Fubini's theorem, one gets,
$$
\int_H\mu_T^x(\d
z)\d\mu_1(x)=\frac{1}{T}\int_0^T\left(\int_HP_x(t,\d z)\d
\mu_1(x)\right)\d t=\frac{1}{T}\int_0^T\mu_1(\d z)\d
t=\mu_1(\d z).
$$
Similarly, we have, $ \int_H\mu_T^y(\d z)\d\mu_2(y)=\mu_2(\d z).$

Now for proving the uniqueness of stationary measures, for $\phi\in
C_b(H)$, we use Fubini's theorem, Jensen's inequality and
stationarity of invariant measures to get,
\begin{eqnarray}\label{qe22}
&&\left|\int_H\phi(z)\d\mu_1(z)-\int_H\phi(z)\d\mu_2(z)\right|
\nonumber\\
&&=\left|\int_H\int_H\phi(z)\mu_T^x(\d
z)\d\mu_1(x)-\int_H\int_H\phi(z)\mu_T^y(\d
z)\d\mu_2(y)\right|\nonumber\\
&&=\left|\frac{1}{T}\int_0^T\left[\int_H\int_H\phi(z)P_x(t,\d
z)\d\mu_1(x)-\int_H\int_H\phi(z)P_y(t,\d
z)\d\mu_2(y)\right]\d t\right|\nonumber\\
&&=\left|\frac{1}{T}\int_0^T\int_H\mathbb{E}\left(\phi(\x^x(t))\right)\d\mu_1(x)\d
t-\frac{1}{T}\int_0^T\int_H\mathbb{E}\left(\phi(\x^y(t))\right)\d\mu_2(y)\d
t\right|\nonumber\\
&&=\left|\frac{1}{T}\int_0^T\left[\int_H\int_H\mathbb{E}\left(\phi(\x^x(t))\right)\d\mu_1\d\mu_2
-\int_H\int_H\mathbb{E}\left(\phi(\x^y(t))\right)\d\mu_2\d\mu_1\right]\d
t\right|\nonumber\\
&&\leq
\frac{1}{T}\int_H\int_H\int_0^T\mathbb{E}\big|\phi(\x^x(t))-\phi(\x^y(t))\big|
\d t\d\mu_1(x)\d \mu_2(y).
\end{eqnarray}
By the exponential stability (Theorem \ref{Them1}) and the
continuity of $\phi$, we have $|\phi(\x^{x}(t))-\phi(\x^{y}(t))|\to
0$ a. s. as $t\to \infty$. Therefore, we get,
$\frac{1}{T}\int_0^T|\phi(\x^{x}(t))-\phi(\x^{y}(t))|\d t\to 0$ as
$T\to\infty$. Hence, by dominated convergence theorem, the last term
in the inequality (\ref{qe22}) tends to $0$ as $T\to \infty$.
\end{proof}

\begin{remark}
Since the abstract functional setting for a class of nonlinear
stochastic hydrodynamic models perturbed by L\'{e}vy noise, namely
\textit{$2D$ Navier-Stokes equations}, \textit{$2D$ Boussinesq model
for the B\'{e}nard convection}, \textit{$2D$ magnetic B\'{e}nard
problem}, \textit{$3D$ Leray $\alpha$-model for Navier-Stokes
equations}, \textit{Shell models of turbulence} are same as that of
\textit{$2D$ magneto-hydrodynamic equations}, the main results
discussed in this paper will hold for these models also.
\end{remark}

\par\bigskip\noindent
{\bf Acknowledgements:} Manil T. Mohan would like to thank Council
of Scientific and Industrial Research (CSIR), India for a Senior
Research Fellowship (SRF). The authors would also like to thank
Indian Institute of Science Education and Research (IISER)-
Thiruvananthapuram for providing stimulating scientific environment
and resources. We thank Prof. P. Sundar for fruitful discussions
about invariant measures during his visit to IISER
Thiruvananthapuram and also bringing our notice to his works
\cite{Su2, Su1} on the subject. The authors would also like to thank
the anonymous referee for his/her valuable comments.

\end{document}